\documentclass [11pt,a4paper, intlimits, reqno]{amsart}
\usepackage{amsmath,amssymb}
\usepackage{amstext, amsthm, ascmac,cases, mathtools}
\usepackage{amscd}
\usepackage{framed}
\usepackage{fancyhdr}
\usepackage[final]{pdfpages}
\usepackage{graphicx,xcolor}

\usepackage{tikz}
\usetikzlibrary{positioning,shapes,arrows}
\usetikzlibrary{3d,calc}
\tikzstyle{na} = [baseline=-.5ex]
\usetikzlibrary{intersections,math,patterns}
\usepackage{pgfplots}
\pgfplotsset{compat=1.6}

\pgfplotsset{soldot/.style={color=blue,only marks,mark=*}}
\pgfplotsset{holdot/.style={color=blue,fill=white,only marks,mark=*}}

\definecolor{cobalt}{rgb}{0.0, 0.28, 0.67}
\definecolor{brightcerulean}{rgb}{0.11, 0.67, 0.84}
 \usepackage[colorlinks=true,linkcolor=cobalt, setpagesize=false, linkbordercolor=cobalt]{hyperref}
\hypersetup{urlcolor=brightcerulean, citecolor=brightcerulean}

\usepackage{amsgen}

\usepackage{mathrsfs}

\usepackage{geometry}
\usepackage{wrapfig}

\usepackage{enumerate}

\DeclareMathOperator*{\esssup}{ess\,sup}

\DeclareMathOperator*{\diam}{diam}

\def\Xint#1{\mathchoice
{\XXint\displaystyle\textstyle{#1}}%
{\XXint\textstyle\scriptstyle{#1}}%
{\XXint\scriptstyle\scriptscriptstyle{#1}}%
{\XXint\scriptscriptstyle\scriptscriptstyle{#1}}%
\!\int\limits}
\def\XXint#1#2#3{{\setbox0=\hbox{$#1{#2#3}{\int}$ }
\vcenter{\hbox{$#2#3$ }}\kern-.6\wd0}}

\def\dashint{\Xint-}

\makeatletter
  \def\thefootnote{\ifnum\c@footnote>\z@\leavevmode\lower.5ex%
      \hbox{$^{\@arabic\c@footnote)}$}\fi}
  \makeatother

\newtheoremstyle{mystyle}
  {}
  {}
  {\itshape}
  {}
  {\bfseries }
  { }
  { }
  {\thmname{#1}\thmnumber{ #2}\thmnote{ (#3)}}

\theoremstyle{mystyle}

\usepackage{lipsum}

\makeatletter
\renewcommand{\@secnumfont}{\bfseries}
\renewcommand\section{\@startsection{section}{2}%
  \z@{-.5\linespacing\@plus-.7\linespacing}{.5\linespacing}%
  {\large\bfseries}}
\renewcommand\subsection{\@startsection{subsection}{3}%
  \z@{.5\linespacing\@plus.7\linespacing}{-.5em}%
  {\normalfont\bfseries}}
 \makeatother

\usepackage[title, toc, titletoc]{appendix}

\allowdisplaybreaks[4]

\newtheorem{thm}{Theorem}[section]

\newtheorem{lem}[thm]{Lemma}
\newtheorem{dfn}[thm]{Definition}

\makeatletter

\@addtoreset{equation}{section}
\makeatother

\renewcommand{\to}{\longrightarrow}
\renewcommand{\rho}{\varrho}

  \renewcommand{\geq}{\geqq}
  \renewcommand{\leq}{\leqq}
   \renewcommand{\rho}{\varrho}

\def \bR{\mathbb R}

\begin{document}
\titlepage

\title{Existence for doubly nonlinear fractional $p$-Laplacian equations}

\date{\textcolor{cyan}{\textbf{20210825 \,\,ver.}}}

\author{Nobuyuki Kato}
\address[Nobuyuki Kato]{Department of Mathematics, Nippon Institute of Technology, Saitama, Japan}
\email{nbykkato@nit.ac.jp}

\author{Masashi Misawa}
\address[Masashi Misawa]{Faculty of Advanced Science and Technology, Kumamoto University, Kumamoto 860-8555, Japan}
\email{mmisawa@kumamoto-u.ac.jp}

\author{Kenta Nakamura}
\address[Kenta Nakamura]{Headquarters for Admissions and Education, Kumamoto University, Kumamoto 860-8555, Japan}
\email{kntkuma21@kumamoto-u.ac.jp}

\author{Yoshihiko Yamaura}
\address[Yoshihiko Yamaura]{Department of Mathematics, College of Humanities and Sciences, Nihon University, Tokyo, Japan}
\email{yamaura@math.chs.nihon-u.ac.jp}

\thispagestyle{firstpage}

\setcounter{page}{1}

 
 %
\keywords{Nonlocal doubly nonlinear equation, Fractional Sobolev space, Space-time fractional Sobolev inequalities}
\medskip
\subjclass[2010]{Primary: 35B45, 35B65, \quad Secondary: 35D30, 35K61}

 \maketitle

 \begin{abstract}
We prove the existence of a weak solution to a doubly nonlinear parabolic fractional $p$-Laplacian equation, which has general doubly nonlinearlity including not only the Sobolev subcritical/critical/supercritical cases but also the slow/fast diffusion ones. Our proof reveals the weak convergence method for the doubly nonlinear fractional $p$-Laplace operator.
\end{abstract}

\tableofcontents

 \section{Introduction and Results}
The aim of this paper is the study of existence of a weak solution to doubly nonlinear parabolic nonlocal equations having the fractional $p$-Laplace operator as the principal part. The double nonlinearity is treated under the general settings including not only the Sobolev subcritical/critical/supercritical cases but also the slow/fast diffusion ones.
%
%
Let $\Omega \subset \mathbb{R}^n$ ($n \geq 1$) be a bounded domain with boundary $\partial \Omega$. We fix indices $s \in (0,1)$, $p>1$ and $q>0$. 
We are interested in the doubly nonlinear nonlocal problem: 
\begin{equation} \label{maineq}
\begin{cases}
\,\,\partial_t(|u|^{q-1}u)+(-\Delta)_p^su=0 &\textrm{in}\quad \Omega_\infty:=\Omega \times (0,\infty), \\
%
\,\,u=0 &\textrm{on}\quad \!\!(\mathbb{R}^n \setminus \Omega)\times(0,\infty) , \\
%
\,\,u(\cdot, 0)=u_0(\cdot ) &\textrm{in}\quad \Omega, 
\end{cases}
\end{equation}
where the initial datum $u_0$ belongs to both the fractional Sobolev space $W^{s,p}_0 (\Omega)$ and the Lebesgue space $L^{q+1}(\Omega)$.
Here we remark that $W^{s,p}_0 (\Omega)$ consists of functions in the fractional Sobolev space $W^{s,p} (\bR^n)$ vanishing outside $\Omega$.
The precise definition of these function spaces will be given in Section~\ref{Sect. 2}.
The fractional $p$-Laplace operator $(-\Delta)_p^s$ is defined by
\begin{align}\label{frac. p-Lap.}
(-\Delta)_p^s u(x,t)&:=\mathrm{P.V.}\int_{\mathbb{R}^n} \frac{|u(x,t)-u(y,t)|^{p-2}(u(x,t)-u(y,t))}{|x-y|^{n+sp}}\,dy \notag\\
&=\lim_{\varepsilon \searrow 0} \int_{\mathbb{R}^n \setminus B_\varepsilon(x)} \frac{|u(x,t)-u(y,t)|^{p-2}(u(x,t)-u(y,t))}{|x-y|^{n+sp}}\,dy,
\end{align}
where the symbol $\mathrm{P.V.}$ means ``in the principal value sense''.

\medskip

Let us introduce our interest in~\eqref{maineq} motivated by an energy structure (refer to~\cite{DiCKP1, DiCKP2, IMS}). For this purpose, we shall simply demonstrate a formal derivation of the equation~\eqref{maineq}$_1$. Consider the integral functional
\[
\mathcal{E}(u):=\frac{1}{2p}\iint_{\bR^n \times \bR^n} \frac{|u(x)-u(y)|^p}{|x-y|^{n+sp}}\,dxdy
\]
for $u \in W^{s, p}_0 (\Omega)$. The G\^{a}teaux differential of $\mathcal{E}(u)$ in the direction $\varphi \in W^{s, p}_0 (\Omega)$ turns out to be
%
\begin{align*}
\frac{d}{d\tau}\mathcal{E}(u+\tau \varphi)\Bigg|_{\tau=0}&= \frac{1}{2p} \iint_{\bR^n \times \bR^n}\frac{d}{d\tau}\,\frac{|u(x)-u(y)+\tau (\varphi(x)-\varphi(y))|^p}{|x-y|^{n+sp}} \,dxdy \,\Bigg|_{\tau=0} \\[2mm]
&=\frac{1}{2}\iint_{\bR^n \times \bR^n} \frac{|u(x)-u(y)|^{p-2}(u(x)-u(y))}{|x-y|^{n+sp}}(\varphi(x)-\varphi(y))\,dxdy \\[2mm]
&=\langle (-\Delta)^s_pu,\,\varphi \rangle,
\end{align*}
where the symbol $\langle \cdot\,, \,\cdot \rangle$ represents the dual pairing between $W^{s,p}_0(\Omega)$ and its dual space $W^{-s,p^\prime}(\Omega)$.
This implies that $(-\Delta)_p^s u$ is the gradient vector field $\nabla \mathcal{E}(u)$ on $W^{s,p}_0(\Omega)$:
\[
\nabla \mathcal{E}(u)=(-\Delta)^s_pu
\]
and hence, the equation~\eqref{maineq}$_1$ is interpreted as a nonlinear generalization of the usual gradient flow equation $\partial_tu(t)=-\nabla \mathcal{E}(u(t))$. We can also refer to \cite{KP} about a variety of interesting problems, their results and the progress of research on the fractional operator.
\medskip

The main result of this paper is the global existence of a weak solution to~\eqref{maineq} as follows:
\begin{thm}\label{mainthm}
Suppose that the initial datum $u_0$  belongs to $W^{s,p}_0(\Omega) \cap L^{q+1}(\Omega)$. Then, there exists a weak solution to~\eqref{maineq} in the sense of Definition~\ref{def. of weak sol.}. The weak solution $u$ of~\eqref{maineq} satisfies the regularity: $\partial_t (|u|^{\frac{q - 1}{2}} u) \in L^2 (\Omega_\infty)$.
\end{thm}
\medskip

The global existence of a weak solution in the fractional Sobolev space as in Theorem ~\ref{mainthm} is shown under the general double nonlinearity. The concentration only on the prototype equation makes it possible to describe clearly the approach for the proof of the existence. The proof is usually based on the weak compactness method with an appropriate approximation. 
A space-time fractional Sobolev inequality is devised and then, induce the weak convergence for the fractional $p$-Laplace operator with combination of the Vitali convergence theorem.
The approximate equation used here is given by the difference integro-differential equations of the so-called Rothe type, associated to~\eqref{maineq} (see~\cite{Rothe}). The approximate solutions are constructed by an application of the direct method in the calculus of variations, which enable us to treat the general double nonlinearity with $ p>1$, $q >0$ and $s \in (0, 1)$ in the equation. 
We shall mention the boundary condition in~\eqref{maineq}. The nonlocal Dirichlet boundary condition~\eqref{maineq}$_2$, which may look weird at a first glance, is suitable and consistent with the nonlocal character of the fractional operator. Here we do \emph{not} require that $C^\infty_0 (\Omega)$ is dense in $W^{s, p}_0 (\Omega)$. In fact, such a density property does not necessarily hold for a general domain without the smoothness of boundary  (see~\cite{FSV} and~\cite[Remark 3.5, p. 533]{Hitchhiker}). Here the equation is considered on any bounded domain $\Omega$ in $\bR^n$ which of the boundary is not necessarily smooth.

\medskip

For the parabolic fractional $p$-Laplace equation of the regular form (Eq.~\eqref{maineq} with $q = 1$), the existence of weak solutions have been studied under weaker conditions for initial data (see \cite{Puhst, Mazon-Rossi-Toledo, Vazquez, AABP}). Here the doubly nonlinear parabolic fractional $p$-Laplace equation~\eqref{maineq} of fairly general form is considered and the  existence of a solution is studied in the energy class of the fractional Sobolev space, as explained before. 
Our proof of the existence in the energy class develops a weak convergence method, based on certain space-time fractional Sobolev inequality and Vitali's convergence theorem.
To the best of our knowledge, that includes a new content for evolutionary fractional differential equations. 
In fact, a truncation technique and a space-time fractional Sobolev inequality are devised and combined with Vitali's convergence theorem to derive the space-time Lebesgue strong convergence. See Lemmata~\ref{truncate energy},~\ref{Cauchy},~\ref{Chebyshev} and~\ref{Cauchyuh} and,  Lemma~\ref{FSineq.II} in Appendix~\ref{Appendix C}. Our truncation argument may be somehow technical and, however, clearly overcomes some difficulties stemming from the power nonlinearity in the time derivative. The Lebesgue strong convergence is naturally applied to take the limit in the nonlinear integral quantities (see the proof of Lemma~\ref{convergence result} for details). It is also shown that the initial datum is continuously attained in the fractional Sobolev space, whose proof can be seen in the proof of (D3).   

\medskip

In our future work we shall consider the gradient flow associated with the best constant of fractional Sobolev inequality, where our approach requires the prototype equation as in~\eqref{maineq} as a fundamental tool. See~\cite{BMS} for the topics of the best constant and its extremal function in the fractional Sobolev inequality. In the usual local setting we have studied the $p$-Sobolev flow, which is the gradient flow concerning the usual Sobolev inequality and based on the prototype doubly nonlinear parabolic equation (refer to~\cite{Nakamura-Misawa,Kuusi-Misawa-Nakamura1,Kuusi-Misawa-Nakamura2}). The related results for PDEs can be seen in~\cite{Alt-Luckhaus, Chen-Hong-Hungerbuhler, DiBenedetto1, Hungerbuhler}.
A certain variational inequality method is well adopted to the doubly nonlinear PDEs in~\cite{BDMS1} with its precise description and comprehensive reference on related equations. Here we shall explain another motive from the geometric perspective: Given a compact Riemannian manifold $(M^n,g_0)$ of dimension $n \geq 2$ and a constant $s \in (0,\frac{n}{2})$, Graham and Zworski (\cite{Graham-Zworski}) originally introduced, in their study of a family of conformally covariant operators, the nonlocal evolutionary equation for a metric $g(t)=u^{\frac{4}{n-2s}}(t)g_0$ on $M$, called a \emph{fractional Yamabe flow}. When $M=\mathbb{S}^n$ in particular, through the stereographic projection $\pi: \mathbb{S}^n \setminus \{\mathrm{north pole}\} \to \mathbb{R}^n$,  the fractional Yamabe flow is written in the form 
 \begin{equation}\label{Yamabe1}
\partial_t \left(u^{\frac{n+2s}{n-2s}}(t)\right)=-(-\Delta)^su(t)+r^g_s u^{\frac{n+2s}{n-2s}}(t) \quad \textrm{in}\,\,\,\bR^n,
 \end{equation}
 where $(-\Delta)^su$ denotes the usual fractional Laplacian defined by~\eqref{frac. p-Lap.} with $p=2$. The quantity $r^g_s:=\dashint_M R^g_s\,dvol_g$ in~\eqref{Yamabe1} is the average of the fractional order curvature $R^g_s$ that coincides with the classical scalar curvature in the case $s=1$.
If $s=1$ specifically, then the fractional Yamabe flow~\eqref{Yamabe1} becomes exactly the classical Yamabe flow introduced by Hamilton (\cite{Hamilton}). The stationary problem for the fractional Yamabe flow is the fractional Yamabe problem, for which the existence of a solution
is shown in~\cite{GQ} under the assumption that $s \in (0,1)$ and that the domain manifold $M$ is locally conformally flat with positive Yamabe constant. Jing and Xiong (\cite{Jing-Xiong}) succeeded in obtaining the global existence for the fractional Yamabe flow~\eqref{Yamabe1} and the asymptotic convergence of~\eqref{Yamabe1} to the stationary solution. Recently, when the domain manifold is locally conformally flat with positive Yamabe constant, Daskalopoulous, Sire and V\'{a}zquez proved the global existence of a smooth solution to the fractional Yamabe flow and its convergence as $t \to \infty$ to a metric of constant scalar curvature metric (see \cite{DSV}). 
In the Euclidean space, the metric is flat and the above curvature conditions are not verified and thus, there may occur some difficulties from the viewpoint of PDEs. In the outstanding result of~\cite{dPQRV}, the existence and uniqueness for more general fractional porous medium equations is proved via the Fourier analysis and $L^1$-semigroup theory. In this paper, we are forced to consider on the bounded domain $\Omega \subset \bR^n$ and take a direct approach dictated by the structure of fractional $p$-Laplacian term inspired by the energy $\mathcal{E}$, as mentioned above. Our results presented here can cover those of the fractional Yamabe flow in the Euclidean case.
\medskip

The paper is organized as follows. In Section~\ref{Sect. 2}, we fix the notation and summarize some fundamental facts involved with the fractional Sobolev space. Section~\ref{Sect. 3} is devoted to constructing approximate solutions of~\eqref{maineq} and some uniform energy estimates for them. The approximate solutions are given by the time-difference integro-differential  equations~\eqref{NE} associated with \eqref{maineq}.
In Section~\ref{Sect. 4}, we derive the convergence of approximate solutions on the basis of the energy estimates in the previous section. In Section~\ref{Sect. 5}, we show Theorem~\ref{mainthm} by the weak convergence method. Appendix~\ref{Appendix A} is devoted to the proof of the existence of approximate solutions to~\eqref{NE} via the direct method in the calculus of variations. In Appendix~\ref{Appendix B}, we show the boundedness of our approximate solutions under the assumption of that of the initial datum. As an application, we deduce a regularity of $\partial_t\left(|u|^{q-1}u\right)$ in the case $q \geq 1$, as seen in Theorem~\ref{mainthm3}. Here we stress that our main result, Theorem~\ref{mainthm} does not require any pointwise boundedness of approximate solutions. In the final Appendix~\ref{Appendix C} we report our key idea, the space-time fractional Sobolev inequality.\bigskip

\noindent
\textbf{Acknowledgements}
\medskip

Masashi Misawa acknowledges the partial support by the Grant-in-Aid for Scientific Research (C) Grant No.21K03330 (2021) at JSPS. The work by Kenta Nakamura was also partially supported by Grant-in-Aid for Young Scientists Grant No.21K13824 (2021) at JSPS. Yoshihiko Yamaura was also partially supported by Grant-in-Aid for Scientific Research (C) Grant No.18K03381 (2018) at JSPS, and by Individual Research Expense of College of Humanities and Sciences at Nihon University for 2020.
%
%
%
%
\section{Preliminaries}\label{Sect. 2}
We shall split this section in three parts: first, we display our notation, then we collect some inequalities used in the course of the paper, and finally we introduce the definition of weak solutions to our problem~\eqref{maineq}.
\subsection{Notation}

In this paper, let $\Omega \subset \mathbb{R}^n\,\,(n \geq 1)$ be a bounded domain with boundary $\partial \Omega$ and denote for simplicity by $\Omega^c:=\bR^n \setminus \Omega$
the outside of $\Omega$.
For $T \in (0,\infty]$,
let $\Omega_T:=\Omega \times (0,T)$ be a space-time cylinder. 
We shall denote in a standard way $B_\rho(x_0):=\left\{ x \in \bR^n: |x-x_0|<\rho\right\}$,  the open ball with radius $\rho>0$ and center $x_0 \in \mathbb{R}^n$. We adopt the convention of writing $B_\rho$ instead of $B_\rho(x_0)$, when the center is origin $0$ of $\bR^n$, or when the center is clear from the context. For a measurable set $A \subset \bR^k$, $k \in \mathbb{N}$, we designate by $|A|$ the usual $k$-dimensional Lebesgue measure of $A$. 
 \medskip

In what follows, we denote by $C$, $C_1$, $C_2, \cdots $ different positive constants in a given context. 
Relevant dependencies on parameters will be emphasized using parentheses, e.g., $C\equiv C(n,p,\Omega)$ means that $C$ depends on $n, p$ and $\Omega$. For the sake of readability, the dependencies of the constants will be often omitted within the chains of estimates.
Furthermore, the equation number $(\,\cdot\,)_\ell$ denotes the $\ell$-th line of the Eq. $(\,\cdot\,)$.

\medskip
 
Now we define some fractional Sobolev spaces (refer to~\cite{Hitchhiker}).
\medskip

For any $p \in [1, \infty)$ and $s \in (0,1)$, the fractional Sobolev space is defined as
\[
W^{s,p}(\mathbb{R}^n):=\left\{v \in L^p(\mathbb{R}^n)\,:\,\frac{|v(x)-v(y)|}{|x-y|^{\frac{n}{p}+s}} \in L^p(\mathbb{R}^n \times \mathbb{R}^n) \right\}
\]
with the finite norm
\begin{equation} \label{Wsp-norm-Rn}
\|v\|_{W^{s,p}(\bR^n)}:=\|v\|_{L^p(\bR^n)}+[v]_{W^{s,p}(\bR^n)},
\end{equation}
the usual $L^p$-norm $\|v\|_{L^p(\bR^n)}$ in $\bR^n$ together with the so-called \emph{Gagliardo semi-norm} in $\bR^n$
\begin{equation}\label{WspRn-norm}
[v]_{W^{s,p}(\mathbb{R}^n)}:=\left( \,\,\iint_{\bR^n\times \bR^n}\frac{|v(x)-v(y)|^p}{|x-y|^{n+sp}}\,dxdy \right)^\frac{1}{p}.
\end{equation}
Analogously, the fractional Sobolev space on a domain $K \subset \bR^n$ is defined as
\[
W^{s,p}(K):=\left\{v \in L^p(K)\,:\,\frac{|v(x)-v(y)|}{|x-y|^{\frac{n}{p}+s}} \in L^p(K\times K) \right\}
\]
with the finite norm
\begin{equation} \label{Wsp-norm-K}
\|v\|_{W^{s,p}(K)}:=\|v\|_{L^p(K)}+[v]_{W^{s,p}(K)}
\end{equation}
and
\begin{equation}\label{WspK-norm}
[v]_{W^{s,p}(K)}:=\left( \,\,\iint_{K\times K}\frac{|v(x)-v(y)|^p}{|x-y|^{n+sp}}\,dxdy \right)^\frac{1}{p}.
\end{equation}
%
We also define
\begin{equation}\label{Wsp0}
W_0^{s,p}(K):=\left\{ u \in W^{s,p}(\bR^n)\,:\,u=0 \,\,\,\textrm{a.e.\,\,in}\,\,\mathbb{R}^n \setminus K \right\}.
\end{equation}
In the context on the Dirichlet boundary condition for a nonlocal operator $(-\Delta)^s_p$ in $K$,
we say that a function $u$ in $W^{s,p}(\bR^n)$ takes zero value on the boundary of $K$ if $u$ belongs to $W^{s,p}_0(K)$.  By the definition~\eqref{Wsp0} with $K=\bR^n$, it can be naturally understood  that $W_0^{s,p}(\bR^n)=W^{s,p}(\bR^n)$. 
\medskip

For $s \in (0,1)$ and $p \in [1,\infty)$, the fractional spaces $W^{s,p}(\bR^n)$, $W^{s,p}(K)$ and $W_0^{s,p}(K)$ are the Banach spaces with finite norms~\eqref{Wsp-norm-Rn},~\eqref{Wsp-norm-K} and $\|\cdot \|_{W_0^{s,p}(K)}:=\|\cdot \|_{L^p(K)}+[\,\cdot\,]_{W^{s,p}(\bR^n)}$, respectively.The dual space of $W^{s,p} (\bR^n)$ is denoted as $(W^{s, p} (\bR^n))^\ast = W^{- s, p^\prime} (\bR^n)$. Similarly, $(W^{s, p}_0 (K))^\ast = W^{- s, p^\prime} (K)$ denotes the dual space of $W^{s, p}_0 (K)$. Here $p^\prime$ denotes the H\"{o}lder conjugate of the exponent $p$.
 \medskip
 
We also address the space-time function spaces. Let $K$ be a domain in $\bR^n$ and $0 \leq t_1, t_2 < \infty$. Let $1 \leq p,q \leq \infty$. Let $\mathcal{X}$ be a Banach space consists of functions defined on $K$. We use the space of Bochner $L^q (t_1, t_2)$-integrable functions $v : (t_1, t_2) \ni t \mapsto v (t) \in \mathcal{X}$, denoted by $L^q (t_1, t_2 ;\mathcal{X})$. Letting $\mathcal{X}$ be the Lebesgue space $L^p (K)$, we have $L^{q}(t_1,t_2\,;\,L^{p}(K))$ with a finite norm
\[
\|v\|_{L^{q}(t_1,t_2\,;\,L^{p}(K))}:=
\begin{cases}
\displaystyle \left(\int_{t_1}^{t_2}\|v(t)\|_p^{q}\,dt \right)^{1/q}\quad &(1 \leq q<\infty) \\
\displaystyle \sup_{t_1 <t < t_2}\|v(t)\|_p\quad &(q=\infty),
\end{cases}\notag
\]
where $\|v(t)\|_{L^p(K)}$  is abbreviated to $\|v(t)\|_p$ for $1 \leq p \leq \infty$  and $\esssup \limits_{t_1 < t < t_2} \|v (t)\|_p$ to $\sup \limits_{t_1 <  t< t_2} \|v (t)\|_p$.
If $p=q$ then we have the identification as $L^p(K \times (t_1,t_2))=L^{p}(t_1,t_2\,;\,L^{p}(K))$. 
%
%
%
%
\medskip

Let $s\in (0,1)$. The choice of $\mathcal{X}=W_0^{s,p}(K)$ implies that
$L^{q}(t_1,t_2\,;\,W_{0}^{s,p}(K))$ is Banach space with a finite norm
\[
\|v\|_{L^q(t_1,t_2\,;\,W_{0}^{s,p}(K))}:=\left(\int_{t_1}^{t_2}\|v(t)\|_{W^{s,p}(K)}^{q} \,dt\right)^{1/q}
\]
whenever $q<\infty$, and
\[
\|v\|_{L^\infty(t_1,t_2\,;\,W_{0}^{s,p}(K))}:=\esssup_{t_1<t<t_2}\|v(t)\|_{W^{s,p}(K)}.
\]
The dual spaces of Bochner spaces $L^q (t_1, t_2; W^{s, p} (\bR^n))$ and $L^q (t_1, t_2; W^{s, p}_0 (K))$ are denoted by $L^q (t_1, t_2; W^{s, p} (\bR^n))^\ast$ and $L^q (t_1, t_2; W^{s, p}_0 (K))^\ast$, respectively. Here $L^q (t_1, t_2; W^{s, p} (\bR^n))^\ast$ is identified with $L^{q^\prime}(t_1, t_2; W^{-s,p^\prime}(\bR^n))$, where $p^\prime$ and $q^\prime$ are the H\"{o}lder conjugate of $p$ and $q$, respectively. Similarly, we identify $L^q (t_1, t_2; W^{s, p}(K))^\ast=L^{q^\prime}(t_1, t_2; W^{-s,p^\prime}(K))$.
\medskip

For studying the parabolic nonlocal equation in~\eqref{maineq}, we shall exploit the space-time fractional Sobolev space. Let $p \in[1,\infty)$ and $s \in (0,1)$. For a domain $K$ in $\bR^n$ and $0<T \leq \infty$, put $K_T:=K\times (0,T)$. The space-time fractional Sobolev space, denoted by $W^{s,p}(K_T)$, consists of functions $v$ in $L^p(K_T)$ having a finite semi-norm 
\begin{equation}
[v]_{W^{s,p}(K_T)}:=\left(\,\,\iint_{K_T \times K_T}\frac{|v(x,t)-v(x^\prime,t^\prime)|^p}{\left(\sqrt{|x-x^\prime|^2+(t-t^\prime)^2}\right)^{n+1+sp}}\,dxdtdx'dt' \right)^\frac{1}{p},
\end{equation}
which is also the Banach space with respect to the norm
\begin{equation}
\|v\|_{W^{s,p}(K_T)}:=\|v\|_{L^p(K_T)}+[v]_{W^{s,p}(K_T)}.
\end{equation}
\smallskip

\subsection{Fundamental tools}
We gather fundamental tools and facts. First, the three inequalities listed in the following lemma are needed throughout this paper and, in particular, used to evaluate the fractional $p$-Laplacian term. 
%
%
%
\begin{lem}[Algebraic inequality]\label{Algs}
For all $\alpha \in (1,\infty)$ there are positive constants $C_j(\alpha)$,\,$j=1,2$, such that, for all $\xi,\,\eta \in \mathbb{R}$,
\begin{align}
||\xi|^{\alpha-2}\xi-|\eta|^{\alpha-2}\eta| \leq C_1(|\xi|+|\eta|)^{\alpha-2}|\xi-\eta|,
\label{Alg1} \\
(|\xi|^{\alpha-2}\xi-|\eta|^{\alpha-2}\eta)(\xi-\eta) \geq C_2 (|\xi|+|\eta|)^{\alpha-2}|\xi-\eta|^2
\label{Algpre}
\end{align}
and, in particular, when $\alpha \geq 2$ 
\begin{equation}\label{Alg2}
(|\xi|^{\alpha-2}\xi-|\eta|^{\alpha-2}\eta)(\xi-\eta) \geq C_2 |\xi-\eta|^\alpha. 
\end{equation}
\end{lem}
%
%
\begin{proof}
The proof of~\eqref{Alg1} and~\eqref{Algpre} can be derived from the proof of~\cite[Lemma 8.3, p.266]{Giusti}. Inequality~\eqref{Alg2} is simply obtained from the Minkowski inequality and~\eqref{Algpre} in~\cite[Lemma 4.4 in Chapter I, p.13]{DiBenedetto1}.
\end{proof}

The following lemma is often used later and the key in some convergences. The proof is due to Vitali's convergence theorem.
\begin{lem}\label{convergence lemma a}
Let $a \in (1,\infty)$ and $a^\prime$ satisfy $\frac{1}{a}+\frac{1}{a^\prime}=1$. Assume $\{f_h\}_{h>0}$ is bounded in $L^a(K_T)$ and $f_h$ converges to $f$ almost everywhere in $K_T$. Then 
\[
\iint_{K_T} f_h \psi \,dxdt \to \iint_{K_T} f \psi \,dxdt \quad \textrm{for}\,\,\,\psi \in L^{a^\prime}(K_T).
\]
In particular, if $\{ f_h \}_{h > 0}$ is bounded in $L^a (K_T)$  and $f_h$ converges to 0 almost everywhere in $K_T$, then for any $b \in [1, a)$,
\[ 
\int_{K_T} |f_h|^b \,dxdt \to 0 \quad \textrm{as}\,\, \,h \searrow 0.
\]
\end{lem}
\begin{proof}
We first verify that $\left\{f_h \psi \right\}_{h >0}$ is uniformly integrable on $K_T$. Let $\Lambda \subset K_T$ be any measurable set. By the boundedness of $f_h$ in $L^a(K_T)$ and H\"{o}lder's inequality, we observe that
\begin{align*}
\iint_{\Lambda} \left|f_h \psi \right|\,dxdt &\leq \left(\iint_{K_T}|f_h|^{a}\,dxdt  \right)^{\frac{1}{a}} \left(\iint_\Lambda |\psi|^{a^\prime}\,dxdt \right)^{\frac{1}{a^\prime}} \notag\\[3mm]
&\leq C\left(\iint_\Lambda |\psi|^{a^\prime}\,dxdt \right)^{\frac{1}{a^\prime}}.
\end{align*}
Due to the absolute continuity of the Lebesgue integral on the right-hand side, for every $\varepsilon>0$, there exists $\delta>0$ such that
\[
\iint_\Lambda |\psi|^{a^\prime}\,dxdt <\varepsilon\quad \textrm{whenever}\quad |\Lambda|<\delta
\]
and thus, $\left\{f_h \psi \right\}_{h >0}$ is uniformly integrable on $K_T$. Furthermore, it follows from the assumption that
\[
f_h \psi \to f\psi \quad \textrm{a.e.\,\,in}\,\,K_T.
\]
Therefore Vitali's convergence theorem gives that
\[
\iint_{K_T} f_h \psi \,dxdt \to \iint_{K_T} f \psi \,dxdt
\]
as $h \searrow 0$. The validity of second statement follows from choosing $\frac{a}{b}$ as $a$, $|f_h|^b$ as $f_h$, 0 as $f$ and $\psi \equiv 1$.  The proof is concluded.
\end{proof}

We now report the fractional Sobolev embedding theorem, which is proved in~\cite[Theorem 6.7, p.557]{Hitchhiker}, \cite{Ponce}.

\begin{lem}[Fractional Sobolev continuous embedding]\label{Sobolev embedding}
Let $s \in (0,1)$ and $p \in [1, \infty)$ be numbers satisfying $sp<n$.
Let  $p^\star_s:=\frac{np}{n-sp}$ be the Sobolev exponent. Then $W^{s,p}(\bR^n)$ is embedded into $L^{p^\star_s}(\bR^n)$:
\begin{equation} \label{embedRn}
\|v\|_{L^{p^\star_s}(\bR^n)} \leq C[v]_{W^{s,p}(\bR^n)}
\end{equation}
for $v \in W^{s,p}(\bR^n)$, where $C$ is the positive constant depending only on $n,s$ and $p$. For any bounded domain $K$, 
$W^{s,p}(K)$ is embedded into $L^{p^\star_s}(K)$:
\begin{equation*} 
\|v\|_{L^r(K)} \leq C \|v\|_{W^{s,p}(K)}
\end{equation*}
for $v \in W^{s,p}(K)$ and all $r \in [1, p^\star_s]$,
where the positive constant $C$ depends on $K$ as well as $n,s,p$ and $r$.
\end{lem}
%
%
%
The next lemma refers to the fractional Sobolev compact embedding,
whose proof is in~\cite[Corollary 7.2, p.561]{Hitchhiker}.
\begin{lem}[Fractional Sobolev compact embedding]\label{compact embedding}
Let $s \in (0,1)$ and $p \in [1,\infty)$ be numbers satisfying $sp<n$ and $p_s^\star$ be as in Lemma~\ref{Sobolev embedding}. Let $K$ be a Lipschitz bounded domain in $\bR^n$. Then, for any $r \in [1,p^\star_s)$,
the Sobolev embedding $W^{s,p}(K) \hookrightarrow L^r(K)$ given in Lemma~\ref{Sobolev embedding} is compact. In other words, if a sequence $\{v_h\}_{h>0} \subset W^{s,p}(K)$ is bounded, then $\{v_h\}_{h>0}$ is pre-compact in $L^r(K)$
for any $r \in [1,p^\star_s)$.
\end{lem}

We also need the space-time fractional Sobolev continuous and compact embedding.
The proof is done similarly as in Lemma \ref{compact embedding}, on the space-time domain with dimension $n + 1$.
\begin{lem}[Space-time fractional Sobolev embedding]\label{space-time Sobolev embedding}
Let $s \in (0,1)$ and $p \in [1,\infty)$ be numbers with $sp<n+1$.
Let $K$ be a bounded domain in $\bR^n$ and $0<T<\infty$.
For any $v \in W^{s,p}(K_T)$, there holds that, for any $\gamma \in [1,\overline{p^\star_s}\,]$ with $\overline{p^\star_s}:=\frac{(n+1)p}{n+1-sp}$, 
\begin{equation} \label{embedKT}
\|v\|_{L^\gamma(K_T)} \leq C \|v\|_{W^{s,p}(K_T)},
\end{equation}
where the positive constant $C$ depends only on $K$, $T$, $s$ $p$ and $n$.
Furthermore, if $K$ is a Lipschitz bounded domain, then 
the space-time Sobolev embedding $W^{s,p}(K_T) \hookrightarrow L^\gamma(K_T)$
is compact for any $\gamma \in [1,\overline{p^\star_s}\,)$.
In other words,
if a sequence $\{v_h\}_{h>0} \subset W^{s,p}(K_T)$ is bounded,
then $\{v_h\}_{h>0}$ is pre-compact in $L^\gamma(K_T)$ for any $\gamma \in [1, \overline{p^\star_s}\,)$.
\end{lem}
\medskip

The following is the fractional Poincar\'{e} inequality available for functions in the fractional Sobolev space $W^{s, p}_0 (K)$. We shall present the proof for completeness (refer to \cite[Lemma 8.1, Theorem 8.2, pp.562--565]{Hitchhiker}).
\begin{lem}[Fractional Poincar\'{e} inequality]\label{t.Poincare}
Let $s \in (0,1)$, $p>1$ and $K$ be a bounded domain in $\bR^n$. Then for any $u \in W^{s,p}_0(K)$ there holds
\begin{equation}\label{e.Poincare}
\|u\|_{L^p(K)} \leq C(n,s,p)(\diam{K})^{sp}[u]_{W^{s,p}(\bR^n)}.
\end{equation}
\end{lem}
\begin{proof}
First we notice that the assertion follows from the Sobolev inequality~\eqref{embedRn} in Lemma~\ref{Sobolev embedding} and H\"older's inequality. Here the direct proof is given, which is the modification of the Campanato semi-norm estimate in the proof of~\cite[Lemma 8.1, pp.562--565]{Hitchhiker}.

Since $u(x)=0$  for $x \in K^c$, letting $d:=\diam K$, one can estimate as
\begin{align*}
[u]^p_{W^{s,p}(\bR^n)}&=\iint_{\bR^n \times \bR^n} \dfrac{|u(x)-u(y)|^p}{|x-y|^{n+sp}}\,dxdy \geq \iint_{K \times K^c} \dfrac{|u(x)-u(y)|^p}{|x-y|^{n+sp}}\,dxdy
\\[2mm]
& \geq \int_K \left(\,\int_{K^c\,\cap\,\{x\in \bR^n\,:\,|x-y| \,\geq \,2d\}}\dfrac{|u(x)-u(y)|^p}{|x-y|^{n+sp}}\,dx\right)\,dy \\[2mm]
&=\int_K \left(\,\int_{K^c\,\cap\,\{x\in \bR^n\,:\,|x-y| \,\geq \,2d\}}\dfrac{|u(y)|^p}{|x-y|^{n+sp}}\,dx\right)\,dy.
\end{align*}
Since for $y\in K$ the fact that $K \subset \{x\in \bR^n:|x-y| < 2d\}$ shows the region of the integral in the above display is $K^c\,\cap\,\{x\in \bR^n:|x-y| \geq 2d\}=\{x\in \bR^n:|x-y| \geq 2d\}$ and so, the change of variable $\rho=|x-y|$ gives that
\begin{align*}
[u]^p_{W^{s,p}(\bR^n)} &\geq \int_K \left(\int_{\{x\in \bR^n:|x-y| \geq 2d\}}\dfrac{|u(y)|^p}{|x-y|^{n+sp}}\,dx\right)\,dy \\[2mm]
&=\int_K\left(\,\int_{2d}^\infty \dfrac{n\alpha_n}{\rho^{1+sp}}\,d\rho \right)|u(y)|^p\,dy \\[2mm]
&=n\alpha_n\dfrac{(2d)^{-sp}}{sp}\int_K|u(y)|^p\,dy=Cd^{-sp}\|u\|_{L^p(K)}^p,
\end{align*}
where $\alpha_n$ denotes the volume of the unit ball in $\bR^n$ and $C\equiv C(n,s,p)$. Thus the proof is complete.
\end{proof}
We conclude this subsection by listing the result retrieved from~\cite[Proposition 2.1, p.524]{Hitchhiker}.
\begin{lem}\label{monotone}
Let $p\geq1$ and $0<s\leq s^\prime<1$, and $K$ be an open set in $\bR^n$. Then there holds
\[
\|u\|_{W^{s,p}(K)} \leq C\|u\|_{W^{s^\prime,p}(K)}
\]
for a measurable function $u$ on $K$ and suitable positive constant $C$ depending only on $n$,$s$ and $p$.
\end{lem}

\subsection{Weak solution}
Finally, we introduce the notion of weak solution in the following. We hereafter abbreviate  
\begin{align}
U(x,y)&:= |u(x)-u(y)|^{p-2} (u(x)-u(y)),
\label{Uxy} \\
U(x,y,t)&:= |u(x,t)-u(y,t)|^{p-2} (u(x,t)-u(y,t))
\label{Utxy}
\end{align}
for any measurable function $u$ defined on $\bR^n$ and $\bR^n \times \bR$,
respectively.
Such a short-hand notation is used in the fractional operator because of space limitation.
\smallskip

\begin{dfn}[Weak solution]\label{def. of weak sol.}\normalfont
Let $\Omega \subset \bR^n$ be a bounded domain
and $T \in (0,\infty]$. Suppose that the initial datum $u_0$ is in $W^{s, p}_0 (\Omega) \cap L^{q+1} (\Omega)$. 
Let $\mathscr{T}$ be the class of test functions defined by 
\[
\mathscr{T} : = \left\{ \varphi \in L^p (0, T\,; W^{s, p}_0 (\Omega)) :
\left.\begin{array}{c}\partial_t \varphi \in L^{q + 1} (\Omega_T), \\[1mm]\varphi (x, 0) = \varphi (x, T)=0 \quad \textrm{a.e.}\,\, x \in \Omega \end{array}\right.
\right\}.
\]
A measurable function $u=u(x,t)$ defined on a space-time region $\bR^n_T:=\bR^n\times (0,T)$ is said to be a \emph{weak solution} to~\eqref{maineq} provided that the following conditions are satisfied:
\begin{enumerate}[(D1)]
\item $u \in L^\infty(0,T\,; W^{s,p}(\bR^n) \cap L^{q+1}(\bR^n))$.\\[-3mm]
\item There holds
\[
-\iint_{\Omega_T}|u|^{q-1}u \cdot\partial_t\varphi\,dxdt
+\frac{1}{2}\int_0^T\iint_{\bR^n \times \bR^n} \frac{U(x,y,t)}{|x-y|^{n+sp}}(\varphi(x,t)-\varphi(y,t))\,dxdydt=0
\]
for every $\varphi \in \mathscr{T}$.

\item $u$ satisfies the Dirichlet boundary condition in the sense that
\[
u(t) \in W^{s,p}_0(\Omega)\quad \textrm{for a.e.}\,\,\,t \in (0,T),
\]
and attains the initial datum $u_0$ continuously in the fractional Sobolev space:
\[
\lim_{t\searrow 0} \|u(t)-u_0\|_{W^{s,p}(\bR^n)} =0.
\]
\end{enumerate}
\end{dfn}

\section{Approximate solutions}\label{Sect. 3}
%
%
This section is devoted to constructing approximate solutions of~\eqref{maineq} and 
deriving the uniform estimate with respect to $h$. Following~\cite[Section 3]{Nakamura-Misawa}, we introduce a family of nonlocal elliptic equations~\eqref{NE} of Rothe type, characterized by the difference quotient in time variable (refer to~\cite{Rothe}).
Let $h$ be a fixed positive number, sent to zero later. Starting from the initial datum $u_0$ in $W_0^{s,p}(\Omega) \cap L^{q+1}(\Omega)$, we shall inductively define $u_m \in W_0^{s,p}(\Omega) \cap L^{q+1}(\Omega)$ for $m\in\mathbb{N}$ as a weak solution in the sense of~\eqref{weak form NE} below to the following nonlocal integro equation 
\begin{align} \label{NE}
\begin{cases}
\displaystyle \frac{|u_m|^{q-1}u_m - |u_{m-1}|^{q-1} u_{m-1}} {h} + (-\Delta)_p^s u_m =0 
\quad &\hbox{in} \,\,\bR^n \\
u_m=0 \quad &\hbox{in} \,\, \Omega^c.
\end{cases}
\end{align}
%
%
%
%
%
Here, $C^\infty_0 (\Omega)$ is a subspace in $W^{s, p}_0 (\Omega)$ and $L^{q + 1} (\Omega)$, respectively, and thus, $W^{s, p}_0 (\Omega) \cap L^{q + 1} (\Omega)$ is nonempty.

The existence of the weak solution $u_m$ to Eq.~\eqref{NE} is guaranteed by the following lemma, whose proof is given in Appendix~\ref{Appendix A}.
\begin{lem}[Existence of solutions for the nonlocal difference equation]\label{existence of NE}
Suppose that \eqref{NE} has solutions $u_k$ for $k=1,2,\ldots,m-1$.
Then there exists a weak solution $u_m \in W_0^{s,p}(\Omega) \cap L^{q+1}(\Omega)$ to~\eqref{NE}.
\end{lem}
%
%
%
%

Let $\{u_m\}_{m\in\mathbb{N}} \subset W^{s,p}_0(\Omega) \cap L^{q+1}(\Omega)$
be a sequence of weak solutions to~\eqref{NE}
constructed by Lemma~\ref{existence of NE}.
Note that the weak solutions satisfy~\eqref{NE} in the following weak sense
\begin{align}\label{weak form NE}
& \int_\Omega \dfrac{|u_m|^{q-1} u_m 
-|u_{m-1}|^{q-1} u_{m-1} }{h}  \, \xi \,dx+\frac{1}{2}\iint_{\bR^n \times \bR^n}
\dfrac{U_m(x,y)(\xi(x)-\xi(y)) }{ 
|x-y|^{n+sp}} \, dxdy =0
\end{align}
for any $\xi \in W_0^{s,p}(\Omega) \cap L^{q+1}(\Omega)$,
where $U_m(x,y)$ is defined along the rule in \eqref{Uxy}:
\[
U_m(x,y):=|u_m(x)-u_m(y)|^{p-2}(u_m(x)-u_m(y)).
\]

Let $h>0$ be an approximation parameter, finally sent to zero.
In terms of the family $\{u_m\}_{m \in \mathbb{N}}$ of solutions to~\eqref{NE}, we define the  six functions defined on $\Omega_\infty$: $\bar{u}_h,\,u_h,\,\bar{v}_h,\,v_h,\,\bar{w}_h$ and $w_h$. Put $t_m:=mh, \,m=0,1,2,\ldots$. For $m = 1, 2, \ldots$, let
%
\begin{equation}\label{approx. sol. for maineq.1}
\begin{cases}
\bar{u}_h(x,t):=
\begin{cases}
u_0(x), &\quad (x,t) \in \Omega \times [- h, 0] \\[1mm]
u_m(x), &\quad (x,t)\in \Omega \times (t_{m-1}, t_m]
\end{cases} \\[10mm]
\bar{v}_h(x,t):=|\bar{u}_h(x,t)|^{q-1}\bar{u}_h(x,t), \qquad (x,t) \times \Omega \in [-h,\infty) \\[3mm]
\bar{w}_h(x,t):=|\bar{u}_h(x,t)|^{\frac{q-1}{2}}\bar{u}_h(x,t), \qquad (x,t) \in \Omega \times [-h,\infty) 
\end{cases}
\end{equation}
and for $(x,t) \in \Omega \times [t_{m-1},t_m]$, $m = 1, 2, \ldots$,
\begin{equation}\label{approx. sol. for maineq.2}
\begin{cases}
u_h(x,t):=\dfrac{t-t_{m-1}}{h}u_m(x)+\dfrac{t_m-t}{h}u_
{m-1}(x), \vspace{3mm}\\
v_h(x,t):=\dfrac{t-t_{m-1}}{h}|u_m(x)|^{q-1}u_m(x)+\dfrac{t_m-t}{h}|u_{m-1}(x)|^{q-1}u_{m-1}(x), \vspace{3mm}\\
w_h(x,t):=\dfrac{t-t_{m-1}}{h}|u_m(x)|^{\frac{q-1}{2}}u_m(x)+\dfrac{t_m-t}{h}|u_{m-1}(x)|^{\frac{q-1}{2}}u_{m-1}(x).
\end{cases}
\end{equation}
We refer to all of six functions $\bar{u}_h$, $\bar{v}_h$, $\bar{w}_h$, $u_h$ , $v_h$ and $w_h$ as \emph{approximate solutions} of~\eqref{maineq}.
By use of these notation, ~\eqref{NE} is formally rewritten as
\[
\partial_t v_h+(-\Delta)_p^s\bar{u}_h=0 \quad \textrm{in}\,\,\,\Omega_\infty,
\]
%
%
%
%
which holds in the sense of distribution as follows:
\begin{align}\label{weak form NE'}
& \iint_{\Omega_T} \partial_tv_h(x,t)\, \varphi(x,t) \,dxdt
+\frac{1}{2}\int_0^T\iint_{\mathbb{R}^n \times \mathbb{R}^n}
\dfrac{\overline{U}_h(x,y,t)}{ 
|x-y|^{n+sp}}(\varphi(x,t)-\varphi(y,t)) \,dxdydt =0
\end{align}
for any positive number $T<\infty$ and all test-functions $\varphi \in L^1\left(0, T \,; W^{s, p}_0 (\Omega) \cap L^{q + 1} (\Omega)\right)$, where 
\[
\overline{U}_h(x,y,t):=|\bar{u}_h(x,t)-\bar{u}_h(y,t)|^{p-2}(\bar{u}_h(x,t)-\bar{u}_h(y,t)).
\]

\medskip

We first gather the elementary estimates for the step on time functions and the linear-interpolated on time ones. These are often used in the proof of convergence later. The proof simply follows from the definition of~\eqref{approx. sol. for maineq.1} and~\eqref{approx. sol. for maineq.2}.
\begin{lem}\label{elementary est.}
Let $\bar{u}_h, u_h, \bar{v}_h, v_h, {\bar w}_h$ and $w_h$ defined as in~\eqref{approx. sol. for maineq.1} and~\eqref{approx. sol. for maineq.2} by use of the family of functions $\{u_m\}$, $m = 0, 1, \ldots$. Let us denote by  $f_m\,\, (m = 0, 1, \cdots)$ each $u_m$,\,$|u_m|^{\frac{q - 1}{2}} u_m$ or $|u_m|^{q - 1} u_m$. Further, denote by $\bar{f}_h$ each of $\bar{u}_h$,  $\bar{w}_h$ or $\bar{v}_h$, and by $f_h$ each of $u_h$, $w_h$ or $v_h$. Then there holds, for any $(x, t) \in \bR^n \times [t_{m-1},t_m]\,\,(m=1,2,\ldots)$,
\begin{equation}\label{1}
\left| f_h (x, t)\right| \leq \frac{t - t_{m - 1}}{h} |f_m (x)|+ \frac{t_m  - t}{h} |f_{m - 1} (x)| ;\end{equation}
\begin{equation}\label{2}
\left|\bar{f}_h (x, t) - f_h (x, t)\right|\leq \frac{t_m - t}{h}\left|f_m (x) - f_{m - 1} (x)\right|.
\end{equation}
In particular, for any $(x, t) \in \bR^n \times [0, \infty)$
\begin{equation}\label{3}
\left| f_h (x, t) \right|  \leq \left|\bar{f}_h(x, t)\right| + \left|\bar{f}_h (x, t - h)\right| ;
 \end{equation}
 \begin{equation}\label{4}
\left| \bar{f}_h (x, t) - f_h (x, t) \right| \leq \left| \bar{f}_h (x, t) - \bar{f}_h (x, t- h)\right|.
 \end{equation}
 \end{lem}
\medskip

We next derive the energy estimates as follows.
\begin{lem}[Energy estimates]\label{energy est. NE}
Let $\bar{u}_h$ and $u_h$ be the approximate solutions of~\eqref{maineq} defined as in~\eqref{approx. sol. for maineq.1} and~\eqref{approx. sol. for maineq.2}.
Then it holds that, for any positive number $T<\infty$,
\begin{equation}\label{energy est. NE1}
\sup_{0<t<T}\int_\Omega |\bar{u}_h(t)|^{q+1}\,dx\leq \int_\Omega |u_0|^{q+1}\,dx,
\end{equation}
\begin{equation}\label{energy est. NE2}
\int_0^T[\bar{u}_h(t)]_{W^{s,p}(\bR^n)}^p\,dt
\leq \frac{2q}{q+1}\int_\Omega |u_0|^{q+1}\,dx, 
\end{equation}
\begin{equation}\label{energy est. NE3}
\iint_{\Omega_T} \bigg(|\bar{u}_h (x,t)|
 + |\bar{u}_h (x,t - h)|\bigg)^{q - 1}|\partial_t u_h(x,t)|^2\,dxdt\leq C[u_0]_{W^{s,p}(\bR^n)}^p,
 \end{equation}
 where the positive constant $C$ depends only on $p$ and $q$, and
 \begin{equation} \label{energy est. NE4}
 \sup_{0<t<T} [\bar{u}_h(t)]_{W^{s,p}(\bR^n)}\leq [u_0]_{W^{s,p}(\bR^n)}.
 \end{equation}
\end{lem}
\begin{proof}
Throughout this proof, we choose an integer $N$ to satisfsy $t_{N-1} <T \leq t_N$.
Testing ~\eqref{weak form NE} by $\xi=h u_{m}$ and summing up from $m=1$ to $k$ ($k=1,2,\dotsc, N$), we have
\begin{align*}
\sum \limits_{m=1}^k \int_\Omega \left(|u_{m}|^{q-1}u_{m}-|u_{m-1}|^{q-1}u_{m-1} \right)u_m\,dx +\frac{1}{2}\sum_{m=1}^k h\iint_{\bR^n \times \bR^n}\dfrac{|u_m(x)-u_m(y)|^p}{|x-y|^{n+sp}}\,dxdy=0
\end{align*}
and thus,
\begin{equation}\label{appenergyeq1}
\sum \limits_{m=1}^k \int_\Omega \left(|u_{m}|^{q-1}u_{m}-|u_{m-1}|^{q-1}u_{m-1} \right)u_m\,dx+\frac{1}{2}\sum_{m=1}^k h [u_m]_{W^{s,p}(\bR^n)}^p=0.
\end{equation}
By Young's inequality, the first term on the left-hand side of~\eqref{appenergyeq1} is estimated as 
\begin{align*}
\sum \limits_{m=1}^k \int_\Omega \left(|u_{m}|^{q-1}u_{m}-|u_{m-1}|^{q-1}u_{m-1} \right)u_m\,dx
&\geq \sum_{m=1}^k \frac{q}{q+1} \int_\Omega \left(|u_m|^{q+1}-|u_{m-1}|^{q+1} \right)\,dx \\[2mm]
&= \frac{q}{q+1}\int_\Omega \left(|u_k|^{q+1}-|u_{0}|^{q+1}\right) \,dx.
\end{align*}
On substitution into~\eqref{appenergyeq1}, we arrive at
\begin{equation}\label{appenergyeq2}
\frac{q}{q+1}\int_\Omega |u_k|^{q+1}\,dx+\frac{1}{2} \sum \limits_{m=1}^k h [u_m]_{W^{s,p}(\bR^n)}^p \leq \frac{q}{q+1}\int_\Omega |u_0|^{q+1}\,dx.
\end{equation}
%
Then we neglect the second term and the first one in the left-hand side of~\eqref{appenergyeq2} to get
\[
\sup_{1\leq k\leq N}\int_\Omega |u_k|^{q+1}\,dx \leq \int_\Omega |u_0|^{q+1}\,dx, \]
and
\[
\int_0^{t_N} [\bar{u}_h(t)]_{W^{s,p}(\bR^n)}^p \,dt \leq \frac{2q}{q+1}\int_\Omega |u_0|^{q+1}\,dx,
\]
respectively.
%
%
Those inequalities easily lead to
\[
\sup_{0<t<T}\int_\Omega |\bar{u}_h(t)|^{q+1}\,dx \leq \int_\Omega |u_0|^{q+1}\,dx
\]
and
\[
\int_0^T[\bar{u}_h(t)]_{W^{s,p}(\bR^n)}^p\,dt \leq \frac{2q}{q+1}\int_\Omega |u_0|^{q+1}\,dx.
\]
Thus estimates~\eqref{energy est. NE1} and~\eqref{energy est. NE2}
are claimed.
%
%
%
\smallskip

On the other hand, we take $\displaystyle \xi=u_{m}-u_{m-1}$ in~\eqref{weak form NE} and sum up the resultant equality from $m=1$ to $k$, where $k=1,2,\cdots,N$. Then we have
\begin{align}\label{appenergyeq4}
&\sum \limits_{m=1}^k \dfrac{1}{h}\int_\Omega (|u_{m}|^{q-1}u_{m}-|u_{m-1}|^{q-1}u_{m-1}) (u_{m}-u_{m-1}) \,dx\notag \\[2mm]
&+\frac{1}{2}\sum\limits_{m=1}^k\,\,\,\iint_{\bR^n\times \bR^n}
\dfrac{U_m(x,y)\big[(u_m(x)-u_{m-1}(x))-(u_m(y)-u_{m-1}(y))\big]}
{|x-y|^{n+sp}} \,dxdy =0.
\end{align}
By~\eqref{Algpre} in Lemma~\ref{Algs} with $\alpha = q + 1$, the integrand of the 1st term on the left-hand side of~\eqref{appenergyeq4} is estimated as 
\begin{align}\label{appenergyeq5}
&\quad \left(|u_{m}|^{q-1}u_{m}-|u_{m-1}|^{q-1}u_{m-1} \right)(u_{m}-u_{m-1}) \notag \\[2mm]
&\geq C(q) \big(|u_m| + |u_{m - 1}|\big)^{q - 1} |u_m - u_{m - 1}|^2.
%
\end{align}
Again, by Young's inequality, we observe that
\begin{align}\label{appenergyeq6}
&\iint_{\bR^n\times \bR^n}
\dfrac{U_m(x,y)\big[(u_m(x)-u_{m-1}(x))-(u_m(y)-u_{m-1}(y))\big]}
{|x-y|^{n+sp}} \,dxdy \notag\\[4mm]
&\geq \iint_{\bR^n \times \bR^n}\dfrac{|u_m(x)-u_m(y)|^p}{|x-y|^{n+sp}}\,dxdy \notag\\
&\quad \quad \quad -\dfrac{1}{p}\iint_{\bR^n \times \bR^n}\dfrac{|u_{m-1}(x)-u_{m-1}(y)|^p}{|x-y|^{n+sp}}\,dxdy -\dfrac{p-1}{p}\iint_{\bR^n \times \bR^n}\dfrac{|u_{m}(x)-u_{m}(y)|^p}{|x-y|^{n+sp}}\,dxdy \notag\\[4mm]
&= \frac{1}{p}\iint_{\bR^n \times \bR^n}\dfrac{|u_m(x)-u_m(y)|^p}{|x-y|^{n+sp}}\,dxdy-\frac{1}{p}\iint_{\bR^n \times \bR^n}\dfrac{|u_{m-1}(x)-u_{m-1}(y)|^p}{|x-y|^{n+sp}}\,dxdy.
\end{align}
Gathering~\eqref{appenergyeq5} and~\eqref{appenergyeq6} in~\eqref{appenergyeq4} yields 
\begin{align}\label{appenergyeq7}
&\sum \limits_{m=1}^k h\int_\Omega C \big(|u_m| + |u_{m - 1}|\big)^{q - 1}\bigg|\frac{u_m-u_{m-1}}{h} \bigg|^2 \,dx \notag \\
&\quad \quad \quad+\frac{1}{2p}\iint_{\bR^n \times \bR^n}\dfrac{|u_k(x)-u_k(y)|^p}{|x-y|^{n+sp}}\,dxdy \leq  \dfrac{1}{2p} \iint_{\bR^n \times \bR^n}\dfrac{|u_0(x)-u_0(y)|^p}{|x-y|^{n+sp}}\,dxdy.
\end{align}
Neglecting either term in the left-hand side, we obtain
\[
C(q)\int_0^{t_k} \int_\Omega \bigg(|\bar{u}_h(x,t)| + |\bar{u}_h(x,t-h)|\bigg)^{q - 1}|\partial_t u_h(x,t)|^2\,dxdt\leq \frac{1}{2p}[u_0]_{W^{s,p}(\bR^n)}^p
\]
and
\[
\max_{1\leq k\leq N} [\bar{u}_h(t)]_{W^{s,p}(\bR^n)}^p\leq [u_0]_{W^{s,p}(\bR^n)}^p,
\]
respectively. Therefore the proof is complete.
%
\end{proof}
The statements of Lemma~\ref{energy est. NE} naturally also holds true for the linear-interpolated solutions by the definition. The proof is clear by~\eqref{1}--\eqref{3} in Lemma~\ref{elementary est.}:
\begin{lem}[Energy estimates]\label{energy est. NE'}
Let $u_h$ be the approximate solution of~\eqref{maineq} defined in~\eqref{approx. sol. for maineq.2}.
Then it holds true that, for any positive number $T<\infty$,
\begin{equation}\label{energy est. NE1'}
\sup_{0<t<T}\int_\Omega |u_h(t)|^{q+1}\,dx\leq \int_\Omega |u_0|^{q+1}\,dx,
\end{equation}
\begin{equation}\label{energy est. NE2'}
\int_0^T[u_h(t)]_{W^{s,p}(\bR^n)}^p\,dt
\leq 2^{p-1}\left(\frac{4q}{q+1}\int_\Omega |u_0|^{q+1}\,dx+h[u_0]_{W^{s,p}(\bR^n)}^p\right), 
\end{equation}
 where the positive constant $C$ depends only on $p$ and $q$, and
 \begin{equation} \label{energy est. NE3'}
 \sup_{0<t<T} [u_h(t)]_{W^{s,p}(\bR^n)}\leq [u_0]_{W^{s,p}(\bR^n)}.
 \end{equation}
\end{lem}
\begin{proof}
The inequalities~\eqref{energy est. NE1'} and~\eqref{energy est. NE3'} follows from the Minkowski inequality,~\eqref{1} in Lemma~\ref{elementary est.} and~\eqref{energy est. NE1}, \eqref{energy est. NE4} in Lemma~\ref{energy est. NE}, respectively.
For the inequality~\eqref{energy est. NE2'}, by~\eqref{3} in Lemma~\ref{elementary est.} and the Minkowski inequality we have
\[
\int_0^T[u_h (t)]_{W^{s, p} (\bR^n)}^p\,dt\leq2^{p-1}\left(\int_0^T[{\bar u}_h (t)]_{W^{s, p} (\bR^n)}^p\,dt + \int_0^T [{\bar u}_h (t - h)]_{W^{s, p} (\bR^n)}^p\,dt
\right),
\]
of which the last terms are applied by~\eqref{energy est. NE2} in Lemma~\ref{energy est. NE} and thus, ~\eqref{energy est. NE2'} follows.
\end{proof}
\begin{lem}[Time-derivative estimate]\label{L2-est.}
Let $v_h$ and $w_h$ be the approximate solutions of~\eqref{maineq} defined by~\eqref{approx. sol. for maineq.2}. 
Then the time derivatives of approximating solutions $\{\partial_t w_h\}_{h>0}$ are bounded in $L^2 (\Omega_T)$ for any positive $T < \infty$, namely,

\begin{equation}\label{(i)}
 \iint_{\Omega_T} |\partial_t w_h|^2\,dxdt \leq C [u_0]_{W^{s,p}(\bR^n)}^p
\end{equation}
and, specifically, if $q \geq 1$, the time derivatives of approximating solutions $\{\partial_t v_h\}$ are bounded in $L^1 (\Omega_T)$ for any positive $T < \infty$, 
\begin{equation}\label{(ii)}
\iint_{\Omega_T} |\partial_t v_h|\,dxdt \leq C|\Omega_T|^{\frac{1}{q+1}}\|u_0\|_{L^{q+1}(\Omega)}^{\frac{q-1}{2}}[u_0]_{W^{s,p}(\bR^n)}^{\frac{p}{2}}\end{equation}
with a positive constant $C=C(p,q)$.
\end{lem}
\begin{proof}
Throughout the proof, we choose an integer $N$ satisfying $t_{N-1} <T \leq t_N$. We now begin the proof of~\eqref{(i)}. 
%
By~\eqref{Alg1} in Lemma~\ref{Algs} with $\alpha = \frac{q+3}{2}$
\begin{equation}\label{wh eq.1}
\left||u_m|^{\frac{q-1}{2}}u_m-|u_{m-1}|^{\frac{q-1}{2}}u_{m-1} \right|^2
\leq C \big(|u_m| + |u_{m - 1}|\big)^{q - 1} |u_m - u_{m - 1}|^2.
\end{equation}
From~\eqref{wh eq.1} and~\eqref{energy est. NE3} in Lemma~\ref{energy est. NE} it follows that, for any $N=1,2,\ldots$
\begin{align*}
\sum_{m=1}^N h \int_\Omega |\partial_tw_h|^2\,dx
&=\sum_{m=1}^Nh \int_\Omega \left|\frac{|u_m|^{\frac{q-1}{2}}u_m-|u_{m-1}|^{\frac{q-1}{2}}u_{m-1}}{h} \right|^2\,dx \\[2mm]
&\!\!\!\stackrel{\eqref{wh eq.1}}{\leq} C\sum_{m=1}^N h\int_\Omega \big(|u_m| + |u_{m - 1}|\big)^{q - 1} \left|\frac{u_m-u_{m-1}}{h} \right|^2\,dx \\[2mm]
&=C\int_0^{t_N}\int_\Omega \big(|\bar{u}_h (x,t)|
 + |\bar {u}_h (x,t-h)|\big)^{q - 1} |\partial_tu_h(x,t)|^2\,dxdt \\[2mm]
&\!\!\stackrel{\eqref{energy est. NE3}}{\leq} C[u_0]_{W^{s,p}(\bR^n)}^p,
\end{align*}
which implies~\eqref{(i)}. 

We shall show the estimation~\eqref{(ii)}. Suppose that $q \geq 1$.
We use H\"{o}lder's inequality twice with pair of exponents
$\left(2, 2\right)$ and $\left(\frac{q+1}{q-1},\,\frac{q+1}{2}\right)$, the energy inequalities~\eqref{energy est. NE1} and~\eqref{energy est. NE3} in Lemma~\ref{energy est. NE} to have 
\begin{align*}
\|\partial_tv_h\|_{L^1(\Omega_T)} &\leq \iint_{\Omega_T} \Big(|\bar{u}_h(t)|+|\bar{u}_h(t-h)|\Big)^{q-1}|\partial_t u_h|\,dxdt \notag\\[3mm]
&\leq \left(\iint_{\Omega_T}\Big(|\bar{u}_h(t)|+|\bar{u}_h(t-h)|\Big)^{q-1}\,dxdt \right)^{\frac{1}{2}} \\
&\quad \quad \quad \quad \quad  \times \left(\iint_{\Omega_T}\Big(|\bar{u}_h(t)|+|\bar{u}_h(t-h)|\Big)^{q-1}|\partial_t u_h|^2\,dxdt \right)^{\frac{1}{2}} \notag\\[3mm]
&\leq |\Omega_T|^{\frac{1}{q+1}}\left(\iint_{\Omega_T}\Big(|\bar{u}_h(t)|+|\bar{u}_h(t-h)|\Big)^{q+1}\,dxdt \right)^{\frac{q-1}{2(q+1)}}\\
&\quad \quad \quad \quad \quad  \times \left(\iint_{\Omega_T}\Big(|\bar{u}_h(t)|+|\bar{u}_h(t-h)|\Big)^{q-1}|\partial_t u_h|^2\,dxdt \right)^{\frac{1}{2}} \notag \\[3mm]
&\!\!\!\!\!\!\!\!\stackrel{\eqref{energy est. NE1},~\eqref{energy est. NE3}}{\leq} C(p,q)|\Omega_T|^{\frac{1}{q+1}}\|u_0\|_{L^{q+1}(\Omega)}^{\frac{q-1}{2}}[u_0]_{W^{s,p}(\bR^n)}^{\frac{p}{2}},
\end{align*}
%
which finishes the proof of~\eqref{(ii)}.
\end{proof}

We conclude this section by deriving the integral estimate for $\bar{u}_h$ and $u_h$, used later in Section~\ref{Sect. 4}.
\begin{lem}\label{energyum}
Let $\bar{u}_h$ and $u_h$ be the approximate solutions to~\eqref{maineq} defined by~\eqref{approx. sol. for maineq.1} and~\eqref{approx. sol. for maineq.2}. Then, for any bounded domain $K$ in $\bR^n$ and any positive $T < \infty$
\begin{equation}\label{energyumeq.0}
\begin{split}
\|\bar{u}_h\|_{L^p (K_T)}^p
&\leq C (\mathrm{diam} \,K)^{sp} \int_0^T [\bar{u}_h(t)]_{W^{s, p}(\bR^n)}^p d t
\leq C(\mathrm{diam} \,K)^{sp}  \|u_0\|_{L^{q + 1} (\Omega)}^{q +1};\\[2mm] 
\|u_h\|_{L^p (K_T)}^p
&\leq
C (\mathrm{diam} \,K)^{sp} \left(\|u_0\|_{L^{q + 1} (\Omega)}^{q +1} +h[u_0]_{W^{s,p} (\bR^n)}^p
\right)
\end{split}
\end{equation}
and
\begin{equation}\label{energyumeq.1}
\begin{split}
\sup_{0 < t < T}\|\bar{u}_h (t)\|_{L^{p_s^\star} (K)}&\leq \widetilde{C}[u_0]_{W^{s,p}(\bR^n)}; \\[2mm]
\sup_{0 < t < T} \|u_h (t)\|_{L^{p_s^\star} (K)} &\leq \widetilde{C}[u_0]_{W^{s,p}(\bR^n)},
\end{split}
\end{equation}
where the $C = C (n, p, q)$  and $\tilde{C}=\tilde{C}(n,s,p)$ are positive constants.
\end{lem}
\begin{proof}
From~\eqref{e.Poincare} in Lemma~\ref{t.Poincare} and~\eqref{energy est. NE2} in Lemma~\ref{energy est. NE} we obtain \eqref{energyumeq.0}$_1$.
Similarly, ~\eqref{energy est. NE2'} in Lemma~\ref{energy est. NE'} yields~\eqref{energyumeq.0}$_2$. By~\eqref{embedRn} in Lemma~\ref{Sobolev embedding} and~\eqref{energy est. NE4} in Lemma~\ref{energy est. NE} one has, for any $t>0$
\begin{align*}
\|\bar{u}_h(t)\|_{L^{p_s^\star}(K)} &\!\!\!\stackrel{\eqref{embedRn}}{\leq}\widetilde{C}(n,s,p) [\bar{u}_h(t)]_{W^{s,p}(\bR^n)} \notag \\[2mm]
&\!\!\!\stackrel{\eqref{energy est. NE4}}{\leq}\widetilde{C}[u_0]_{W^{s,p}(\bR^n)},
\end{align*}
which together with~\eqref{1} in Lemma~\ref{elementary est.} shows~\eqref{energyumeq.1}$_1$. The estimate~\eqref{energyumeq.1}$_2$ is shown similarly.
\end{proof}
\section{Convergence of approximate solutions}\label{Sect. 4}
This section is devoted to deriving the convergence of approximate solution
of~\eqref{maineq}
defined by~\eqref{approx. sol. for maineq.1}
and~\eqref{approx. sol. for maineq.2}. The initial value $u_0$ is assumed
to belong to $W_0^{s,p}(\Omega) \cap L^{q+1}(\Omega)$.

To achieve the convergence of approximate solution
of~\eqref{maineq}, as stated in Lemma~\ref{convergence result} below, we need to introduce the following truncated function: For the approximate solution $\{u_m(x)\}$ to~\eqref{NE} and $2\leq \ell \in \mathbb{N}$, let us define as

\begin{equation*}
(u_m(x))_\pm^{(\ell)}:=\min\left\{\max \left\{(u_m(x))_\pm,\,\dfrac{1}{\ell}\right\},\,\ell \right\},
\end{equation*}
where $(u_m(x))_+:=\max\{u_m(x),0\}$ and $(u_m(x))_-:=(-u_m(x))_+$.
\begin{figure}[h]
\begin{tikzpicture}[domain=-3:2, samples=70, very thick,scale=1.1]
\draw[thin, ->] (-2.5,0)--(4.5,0) node[right] {$u_m$}; 
\draw[thin, ->] (0,-1.2)--(0,1.8) node[above] {\footnotesize$y$}; 
\draw [domain=0.2:1.5]plot(\x, {1*\x}); 
\draw [thin, domain=0:2]plot(\x, {1*\x}); 
\draw (4.5,1.2)node[above] {\footnotesize $y=(u_m)_+^{(\ell)}$};
\draw (-0.3,0.2)node[above]{\footnotesize $1/\ell$};
\filldraw[very thick](-2.0, 0.2)--(0.2,0.2);
\filldraw[very thick](1.5, 1.5)--(3.5,1.5);
\draw[densely dotted] (0,1.5) -- (1.5,1.5);
\draw (-0.3,1.5)node {\footnotesize$\ell$};
\end{tikzpicture}
\caption{Graph of $(u_m(x))_+^{(\ell)}$}
\end{figure}
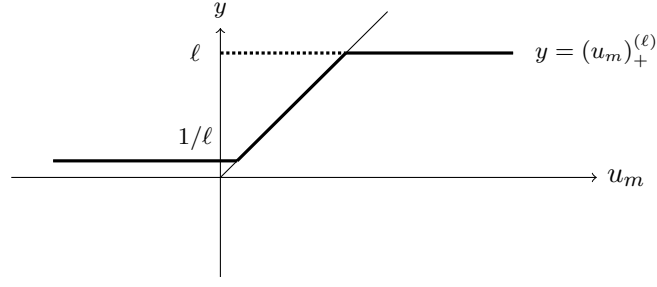
\noindent
Further, for $(x,t) \in \bR^n \times (t_{m-1},t_m] \,\,(m = 1, 2, \cdots)$, we define $\{(u_h)_{\pm}^{(\ell)}\}$ and $\{(\bar{u}_h)_{\pm}^{(\ell)}\}$ by
\begin{equation}\label{truncate}
(u_h)_{\pm}^{(\ell)}(x,t):=\dfrac{t-t_{m-1}}{h}(u_m(x))_{\pm}^{(\ell)}+\dfrac{t_m-t}{h}(u_{m-1}(x))_{\pm}^{(\ell)}
\end{equation}
and
\[
(\bar{u}_h)_\pm^{(\ell)}(x,t):=(u_m(x))_\pm^{(\ell)},
\]
respectively. By $u_m=u_{m-1}=0$ outside $\Omega\,\,(m = 1, 2, \cdots)$,  for any $t \in [0, \infty)$,
\begin{equation}\label{truncate*}
(u_h)_{\pm}^{(\ell)}(x,t)\equiv \dfrac{1}{\ell} \quad \mathrm{outside}\,\,\Omega.
\end{equation}
Sending $\ell \to \infty$, for $(x, t) \in \bR^n \times [t_{m-1}, t_m]\,\,(m = 1, 2, \ldots)$,
\[
(u_h)_{\pm}^{(\ell)}(x,t) \to \dfrac{t-t_{m-1}}{h}(u_m(x))_{\pm}+\dfrac{t_m-t}{h}(u_{m-1}(x))_{\pm}.
\]
Also, the time-derivative of $(u_h)_{\pm}^{(\ell)}$ is
\[
\partial_t(u_h)_{\pm}^{(\ell)}=\dfrac{(u_m(x))_{\pm}^{(\ell)}-(u_{m-1}(x))_{\pm}^{(\ell)}}{h}
\]
for $(x, t) \in \bR^n \times [t_{m-1}, t_m]\,\,(m = 1, 2, \ldots)$.
\subsection{Energy estimates for the truncated solution}
First we deduce certain energy estimate of the truncated solution defined in~\eqref{truncate} above.
\begin{lem}\label{truncate energy}
Let $(u_h)_\pm^{(\ell)}$ be the function defined by~\eqref{truncate}. Then there holds 
\begin{equation}\label{truncate energy1}
\iint_{\Omega_T} |\partial_t (u_h)_\pm^{(\ell)}|^2\,dxdt \leq C\ell^{q-1}[u_0]_{W^{s,p}(\bR^n)}^{p}
\end{equation}
whenever $q \geq 1$ and,
\begin{equation}\label{truncate energy2}
\min \left\{\iint_{\Omega_T} |\partial_t (u_h)_\pm^{(\ell)}|^2\,dxdt, \,\iint_{\Omega_T} |\partial_t (u_h)_\pm^{(\ell)}|^{q+1}\,dxdt \right\} \leq C\ell^{1-q}[u_0]_{W^{s,p}(\bR^n)}^{p}
\end{equation}
whenever $0<q<1$, where the positive constant $C$ depends only on $p$ and $q$.
\end{lem}
\begin{proof}
The proof is based on a sequence of elementary algebraic integral estimates, those are separately performed in several cases. We prove the statement for $\{(u_h)_+^{(\ell)}\}$ only, as the other case is similarly treated. 
To begin, we employ the energy estimate~\eqref{appenergyeq7} in the proof of Lemma~\ref{energy est. NE}
\begin{equation}\label{truncate energy eq.1}
\sum \limits_{m=1}^N h\int_\Omega \big(|u_m| + |u_{m - 1}|\big)^{q - 1}\bigg|\frac{u_m-u_{m-1}}{h} \bigg|^2 \,dx \leq  C(p,q) [u_0]_{W^{s,p}(\bR^n)}^p.
\end{equation}
We use a short-hand notation $[u_m  \geq L] \equiv \Omega \cap \{u_m \geq  L\}$ in the sequel. Other symbols are also abbreviated similarly. Now, we distinguish between the cases $q \geq 1$ and $0 <q<1$. 

In the first case $q \geq 1$, for any $m=1,\ldots ,N$, 
\begin{align}\label{truncate energy eq.2}
\int_\Omega \big(|u_m| + |u_{m - 1}|\big)^{q - 1}\bigg|\frac{u_m-u_{m-1}}{h} \bigg|^2 \,dx &\geq \int_\Omega \big(|u_m| + |u_{m - 1}|\big)^{q - 1}\bigg|\frac{(u_m)_+^{(\ell)}-(u_{m-1})_+^{(\ell)}}{h} \bigg|^2 \,dx \notag\\[3mm]
& =\int_{[u_m >\frac{1}{\ell}]\,\cup\,[u_{m-1}>\frac{1}{\ell}]}\ldots\,dx+\int_{[u_m \leq \frac{1}{\ell}]\,\cap\,[u_{m-1} \leq \frac{1}{\ell}]}\ldots\,dx \notag\\[3mm]
&\geq \int_{[u_m >\frac{1}{\ell}]\,\cup\,[u_{m-1}>\frac{1}{\ell}]}\left(\frac{1}{\ell}\right)^{q-1}\bigg|\frac{(u_m)_+^{(\ell)}-(u_{m-1})_+^{(\ell)}}{h} \bigg|^2 \notag\\[3mm]
&=\int_{[u_m >\frac{1}{\ell}]\,\cup\,[u_{m-1}>\frac{1}{\ell}]}\left(\frac{1}{\ell}\right)^{q-1}\bigg|\frac{(u_m)_+^{(\ell)}-(u_{m-1})_+^{(\ell)}}{h} \bigg|^2\,dx \notag \\[3mm]
&\quad \quad+\int_{[u_m \leq \frac{1}{\ell}]\,\cap\,[u_{m-1} \leq \frac{1}{\ell}]}\left(\frac{1}{\ell}\right)^{q-1}\bigg|\frac{(u_m)_+^{(\ell)}-(u_{m-1})_+^{(\ell)}}{h} \bigg|^2\,dx\notag\\[3mm]
&=\left(\frac{1}{\ell}\right)^{q-1}\int_\Omega \bigg|\frac{(u_m)_+^{(\ell)}-(u_{m-1})_+^{(\ell)}}{h} \bigg|^2\,dx,
\end{align}
where, in the previous line of~\eqref{truncate energy eq.2}, we used the fact that $\frac{(u_m)_+^{(\ell)}-(u_{m-1})_+^{(\ell)}}{h}=0$ on the set $[u_m \leq \frac{1}{\ell}]\,\cap\,[u_{m-1} \leq \frac{1}{\ell}]$. Hence, by~\eqref{truncate energy eq.1} and~\eqref{truncate energy eq.2}, we arrive at the first desired estimate~\eqref{truncate energy1}.

In the latter case $0<q<1$, we need to show
\begin{equation}\label{truncate energy eq.3}
 \big(|u_m| + |u_{m - 1}|\big)^{q - 1}\Big|\tfrac{u_m-u_{m-1}}{h} \Big|^2 \geq 3^{q-1}\min \left\{ \ell^{q-1}\Big|\tfrac{(u_m)_+^{(\ell)}-(u_{m-1})_+^{(\ell)}}{h} \Big|^2,\,\Big|\tfrac{(u_m)_+^{(\ell)}-(u_{m-1})_+^{(\ell)}}{h} \Big|^{q+1} \right\}.
\end{equation}
For the proof of~\eqref{truncate energy eq.3}, a distinction must be made among the cases in the following: Firstly, on the set $\Omega_1:=[u_m \leq \frac{1}{\ell}] \cap [u_{m-1} \leq \frac{1}{\ell}]$ or $\Omega_2:=[u_m> \ell] \cap [u_{m-1} >\ell]$, it holds $\frac{(u_m)_+^{(\ell)}-(u_{m-1})_+^{(\ell)}}{h} =0$ and thus,~\eqref{truncate energy eq.3} is clearly valid. On the set $\Omega \setminus (\Omega_1 \cup \Omega_2)$, we distinguish among the four cases (I)--(IV).
\smallskip

\noindent
\textbf{Case (I)}\,\,$[u_m \leq \frac{1}{\ell}] \cap [\frac{1}{\ell} <u_{m-1} \leq \ell]$: In this region, we further consider the two cases separately: In the case $u_m>0$, noting carefully that $u_m-u_{m-1}<0$, we have
\begin{align*}
|u_m|+|u_{m-1}|=u_m+u_{m-1}&=(u_m-u_{m-1})+2u_{m-1} \\
 &\leq 2u_{m-1} \leq 2\ell <3\ell.
\end{align*}
In the remaining case $u_m \leq 0$, we plainly get
\begin{equation*}
|u_m|+|u_{m-1}|=-u_m+u_{m-1} \leq |u_m-u_{m-1}|.
\end{equation*}
\smallskip

\noindent
\textbf{Case (II)}\,\, $[u_m \leq \frac{1}{\ell}] \cap [u_{m-1}>\ell]$: Also, we distinguish between the two cases $u_m>0$ and $u_m \leq 0$. In the first case $u_m>0$, noting carefully that $u_{m-1}-u_m \geq \ell-\frac{1}{\ell} \geq 2-\frac{1}{2} \geq \frac{1}{\ell}$, we find
\begin{align*}
|u_m|+|u_{m-1}|=(u_{m-1}-u_m)+2u_m &\leq (u_{m-1}-u_m)+\frac{2}{\ell} \\
&\leq (u_{m-1}-u_m)+2(u_{m-1}-u_m) \\
&=3(u_{m-1}-u_m) \\
&\leq 3|u_m-u_{m-1}|.
\end{align*}
In the latter case $u_m \leq 0$, we have again,
\begin{equation*}
|u_m|+|u_{m-1}|=-u_m+u_{m-1} \leq |u_m-u_{m-1}|.
\end{equation*}
\smallskip

\noindent
\textbf{Case (III)}\,\, $[\frac{1}{\ell}<u_m \leq \ell] \cap [\frac{1}{\ell}<u_{m-1}\leq\ell]$: On this set, it clearly holds that
\[
|u_m|+|u_{m-1}|=u_m+u_{m-1} \leq 2\ell.
\]
\smallskip

\noindent
\textbf{Case (IV)}\,\, $[\frac{1}{\ell}<u_m \leq \ell] \cap [u_{m-1}>\ell]$: In this region, we further distinguish between the two cases $u_{m-1}-u_m \leq \ell$ and $u_{m-1}-u_m >\ell$. In the first case $u_{m-1}-u_m \leq \ell$, we in turn see that
\begin{equation*}
|u_m|+|u_{m-1}|=(u_{m-1}-u_m)+2u_m \leq \ell+2\ell=3\ell.
\end{equation*}
In the latter case $u_{m-1}-u_m >\ell$, we have
\begin{align*}
|u_m|+|u_{m-1}|=(u_{m-1}-u_m)+2u_m &\leq (u_{m-1}-u_m)+2\ell \\
& \leq (u_{m-1}-u_m)+2(u_{m-1}-u_m) \\
&=3(u_{m-1}-u_m) \\
&=3|u_m-u_{m-1}|.
\end{align*}
By exchanging of the role of $u_m(x)$ and $u_{m-1}(x)$, we can get the estimates same as Cases (I), (II) and (IV) on the sets $[\frac{1}{\ell}<u_m\leq \ell] \cap [u_{m-1} \leq \frac{1}{\ell}]$,\,$[u_m>\ell] \cap [u_{m-1} \leq \frac{1}{\ell}]$ and $[u_m> \ell] \cap [\frac{1}{\ell}<u_{m-1} \leq \ell]$, respectively. Joining all cases leads to
\begin{equation*}
|u_m|+|u_{m-1}| \leq \max\left\{3\ell, \,3|u_m-u_{m-1}| \right\} \quad \mathrm{in} \quad \Omega \setminus (\Omega_1 \cup \Omega_2),
\end{equation*}
which, by noticing $q-1<0$, in turn implies that
\begin{equation}\label{truncate energy eq.4}
\big(|u_m|+|u_{m-1}|\big)^{q-1} \geq 3^{q-1}\min\left\{\ell^{q-1}, \,|u_m-u_{m-1}|^{q-1} \right\} \quad \mathrm{in} \quad \Omega \setminus (\Omega_1 \cup \Omega_2).
\end{equation}
Using $h^{q-1} \geq 1$ and~\eqref{truncate energy eq.4}, we obtain on the set $\Omega \setminus (\Omega_1 \cup \Omega_2)$, 
\begin{align*}
\big(|u_m| + |u_{m - 1}|\big)^{q - 1}\Big|\tfrac{u_m-u_{m-1}}{h} \Big|^2 &\geq 3^{q-1}\min\left\{\ell^{q-1}\left|\tfrac{u_m-u_{m-1}}{h}\right|^2, \,h^{q-1}\left|\tfrac{u_m-u_{m-1}}{h}\right|^{q+1} \right\} \notag\\[3mm]
&\geq 3^{q-1}\min \left\{ \ell^{q-1}\Big|\tfrac{(u_m)_+^{(\ell)}-(u_{m-1})_+^{(\ell)}}{h} \Big|^2,\,\Big|\tfrac{(u_m)_+^{(\ell)}-(u_{m-1})_+^{(\ell)}}{h} \Big|^{q+1} \right\}.
\end{align*}
In the sequel, we actually have~\eqref{truncate energy eq.3} in the whole region $\Omega$ and thus, by use of~\eqref{truncate energy eq.1} and~\eqref{truncate energy eq.3}, we in turn achieve~\eqref{truncate energy2}. The proof is concluded.
\end{proof}

\subsection{Convergence results for the truncated solution}\label{Sect. Convergence results for the truncated function}
We list here the convergence results, Lemmata~\ref{Cauchy} and~\ref{Cauchyuh} for the truncated solution in~\eqref{truncate}, used later in Lemma~\ref{convergence result}.
\begin{lem}\label{Cauchy}
Let $u_h$ be the approximate solutions of~\eqref{maineq} defined by~\eqref{approx. sol. for maineq.1}. For a given $s \in (0,1)$, let $s^\prime <s$ be a positive number. Let $\ell \geq 2$ be a fixed natural number. Let $K$ be a Lipschitz bounded domain in $\bR^n$. Then there exists a subsequence $\{(u_h)_\pm^{(\ell)}\}$, denoted by the same notation,  and limit functions $\omega_\ell^{\pm}$ such that for every power $\gamma$ with $1\leq \gamma < \frac{n+1}{n+1-s^\prime}$, 
\begin{equation}\label{lemma 4.3 eq.1'}
(u_h)_+^{(\ell)} \to \omega_\ell^{\pm}\quad \mathrm{strongly\,\,in}\quad L^\gamma(K_T).
\end{equation}
\end{lem}
\begin{proof}
We prove the statement for $\{(u_h)^{(\ell)}_+\}$ only, as the other case is similarly shown. Let $2 \leq \ell \in \mathbb{N} $ and $s \in (0,1)$ be arbitrarily given and fixed. Firstly, by ~\eqref{FSineq.2} in Lemma~\ref{FSineq.II}, for all $s^\prime, \bar{s}$ satisfying $0<s^\prime<\bar{s}<s<1$ there holds
\begin{equation}\label{lemma 4.3 eq.0}
\left[(u_h)_+^{(\ell)}\right]_{W^{s^\prime,1}(K_T)} \leq C \left(\left\|\partial_t(u_h)_+^{(\ell)}\right\|_{L^1(K_T)} +\int_0^T \left[(u_h)_+^{(\ell)}(\cdot,t)\right]_{W^{\bar{s},1}(K)}\,dt \right).
\end{equation}
By the energy estimates~\eqref{truncate energy1} and~\eqref{truncate energy2} in Lemma~\ref{truncate energy}, the first term on the right-hand side of~\eqref{lemma 4.3 eq.0} is bounded uniformly in $h$. The Gagliardo semi-norm appearing on the right-hand side of~\eqref{lemma 4.3 eq.0} is estimated as
\begin{align}\label{lemma 4.3 eq.00}
\left[(u_h)_+^{(\ell)}(\cdot,t)\right]_{W^{\bar{s},1}(K)}
&\leq \iint_{K \times K}\dfrac{\left|(u_h)_+(x,t)-(u_h)_+(y,t) \right|}{|x-y|^{(n+sp)\frac{1}{p}}}\cdot \dfrac{1}{|x-y|^{n+\bar{s}-(n+sp)\frac{1}{p}}}\,dxdy  \notag\\[3mm]
&\leq \left( \,\,\,\iint_{K \times K}\dfrac{\left|(u_h)_+(x,t)-(u_h)_+(y,t) \right|^p}{|x-y|^{n+sp}}\,dxdy \right)^{\frac{1}{p}} \notag \\[3mm]
&\quad \quad \quad \times \left( \,\,\,\iint_{K \times K}\dfrac{1}{|x-y|^{\frac{p}{p-1}\left(n+\bar{s}-(n+sp)\frac{1}{p}\right)}}\,dxdy \right)^{\frac{p-1}{p}}.
\end{align}
Noticing that $\frac{p}{p-1}\left(n+\bar{s}-(n+sp)\frac{1}{p}\right)=n+\frac{p}{p-1}(\bar{s}-s)<n$ and letting $R:=\mathrm{diam}\,K$ we compute that
\begin{align*}
\iint_{K \times K}\dfrac{1}{|x-y|^{\frac{p}{p-1}\left(n+\bar{s}-(n+sp)\frac{1}{p}\right)}}\,dxdy &\leq \int_K \left(\,\int_{B_R(x)} \dfrac{dy}{|x-y|^{n+\frac{p}{p-1}(\bar{s}-s)}} \right)\,dx \\[3mm]
&=C(n)|K|\int_0^R\dfrac{d \rho}{\rho^{1+\frac{p}{p-1}(\bar{s}-s)}} \\[3mm]
&=C(n)|K|\dfrac{p-1}{p(s-\bar{s})}R^{\frac{p}{p-1}(s-\bar{s})}
\end{align*}
and thus, the second integral on the right-hand side of~\eqref{lemma 4.3 eq.00} is finite. We therefore obtain from the energy estimate~\eqref{energy est. NE2} in Lemma~\ref{energy est. NE} that
\begin{align*}
\left[(u_h)_+^{(\ell)}(\cdot,t)\right]_{W^{\bar{s},1}(K)} &\leq  C(n,p,\bar{s},s) \left[(u_h)_+(\cdot,t)\right]_{W^{\bar{s},p}(\bR^n)} \notag\\[3mm]
&\leq C(n,p,\bar{s},s) \left[u_0\right]_{W^{\bar{s},p}(\bR^n)},
\end{align*}
which, together this with~\eqref{lemma 4.3 eq.0} and the fact that $\left\| (u_h)_+^{(\ell)}(\cdot,t) \right\|_{L^1(K)} \leq \ell|K_T|$, in turn yields that $\left\{(u_h)_+^{(\ell)}\right\}_{h>0}$ is bounded  in $W^{s^\prime,1}(K_T)$. Hence, by~Lemma~\ref{space-time Sobolev embedding}, for all $\gamma$ with $1 \leq \gamma<\frac{n+1}{n+1-s^\prime}$, there exist a subsequence $\{(u_h)_+^{(\ell)}\}_{h>0}$ (also labelled with $h$) and a limit function $\omega_\ell^+$ fo each $\ell$ such that
\begin{equation*}\label{lemma 4.3 eq.1}
(u_h)_+^{(\ell)} \to \omega_\ell^+\quad \mathrm{strongly\,\,in}\quad L^\gamma(K_T)
\end{equation*}
as $h \searrow 0$. Therefore the proof of Lemma~\ref{Cauchy} is concluded.
~\end{proof}
We shall make a Chebyshev type estimate for the truncated solutions:
\begin{lem}\label{Chebyshev}
Let $\{u_m\}$ be the approximate solution to~\eqref{NE}. Let $\ell \geq 2$  and $\gamma \in [1, p_s^\star)$ be taken arbitrarily. Then there holds that
\begin{equation}\label{lemma 4.3 eq.3}
\left| K \cap \{(u_m)_+ \geq \ell \} \right| \leq \dfrac{C}{\ell^{p_s^\star}}[u_0]_{W^{s,p}(\bR^n)}^{p_s^\star}
\end{equation}
and
\begin{equation}\label{lemma 4.3 eq.4}
\left\|(u_m)_+\right\|_{L^\gamma(K \,\cap\, \{(u_m)_+ \geq \,\ell \})} \leqq \dfrac{C}{\ell^{\frac{p_s^\star}{\gamma}-1}}[u_0]_{W^{s,p}(\bR^n)}^{\frac{p_s^\star}{\gamma}}.
\end{equation}
\end{lem}
\begin{proof}
We report the short proof. By~\eqref{energyumeq.1}$_1$ in~Lemma~\ref{energyum}, one has
\begin{align*}
\ell^{p_s^\star}\left| K \cap \{(u_m)_+ \geq \ell \} \right| \leq  \int_K (u_m)_+^{p_s^\star}\,dx &\leq \|u_m\|_{L^{p_s^\star}(K)}^{p_s^\star} \leq C[u_0]_{W^{s,p}(\bR^n)}^{p_s^\star},
\end{align*}
which in turn implies~\eqref{lemma 4.3 eq.3}. 
By H\"{o}lder's inequality, ~\eqref{lemma 4.3 eq.3} and~\eqref{energyumeq.1} in~Lemma~\ref{energyum}, one estimates that
\begin{align*}
\left\|(u_m)_+\right\|_{L^\gamma(K \,\cap\, \{(u_m)_+ \geq \,\ell \})}^\gamma &\leq \left(\,\int_{K \,\cap\, \{(u_m)_+ \geq \,\ell \}} |(u_m)_+|^{p_s^\star}\,dx \right)^{\frac{\gamma}{p_s^\star}} \left| K \cap \{(u_m)_+ \geq \ell \} \right|^{1-\frac{\gamma}{p_s^\star}} \notag\\[3mm]
&\leq C\dfrac{[u_0]_{W^{s,p}(K)}^{p_s^\star}}{\ell^{p_s^\star-\gamma}},
\end{align*}
which is exactly~\eqref{lemma 4.3 eq.4}.
\end{proof}

We now will show that $\{(u_h)_{\pm}\}_{h>0}$ is a Cauchy sequence in $L^\gamma(K_T)$ in the next lemma.

\begin{lem}\label{Cauchyuh}
$\{(u_h)_{\pm}\}_{h>0}$ is a Cauchy sequence in $L^\gamma(K_T)$ for $1 \leq \gamma < \min \left\{
p_s^\star, \frac{n + 1}{n + 1 - s^\prime} \right\}$ and $0 <s^\prime < s (< 1)$.
\end{lem}
\begin{proof}
We prove the statement for $\{(u_h)^{(\ell)}_+\}$ only, as the other case is similarly treated. We start with dividing into three terms:
\begin{align}\label{lemma 4.3 eq.2}
&\left\|(u_h)_+-(u_{h'})_+\right\|_{L^\gamma(K_T)} \notag\\[2mm]
&\leq \left\|(u_h)_+-(u_h)_+^{(\ell)}\right\|_{L^\gamma(K_T)} +\left\|(u_h)_+^{(\ell)}-(u_{h'})_+^{(\ell)}\right\|_{L^\gamma(K_T)} +\left\|(u_{h'})_+^{(\ell)}-(u_{h'})_+\right\|_{L^\gamma(K_T)}  \notag\\[2mm]
&:=\mathbf{I}_{h,\ell}+\mathbf{I}_{h,h',\ell}+\mathbf{I}_{h',\ell}.
\end{align}

From now on, we shall estimate each term $\mathbf{I}_{h,\ell},\,\mathbf{I}_{h,h',\ell}$ and $\mathbf{I}_{h',\ell}$ in~\eqref{lemma 4.3 eq.2}. Since for $(x,t) \in K \times (t_{m-1},t_m)\,\,(m =1, 2, \cdots)$
\begin{align*}
&(u_h)_+(x,t)-(u_h)_+^{(\ell)}(x,t) \\
&=\dfrac{t-t_{m-1}}{h}\left((u_m)_+(x)-(u_m)_+^{(\ell)}(x) \right)+\dfrac{t_m-t}{h}\left((u_{m-1})_+(x)-(u_{m-1})_+^{(\ell)}(x)\right)
\end{align*}
By the Minkowski inequality, $\mathbf{I}_{h,\ell}$ is estimated as
\begin{align*}
\mathbf{I}_{h,\ell}&\leq\left(h \sum_{m = 1}^N \left\|(u_m)_+ - (u_m)_+^{(\ell)}\right\|_{L^\gamma (K)}^\gamma \right)^{\frac{1}{\gamma}}
+
\left( h \sum_{m = 1}^N \left\| (u_{m - 1})_+ - (u_{m - 1})_+^{(\ell)} \right\|_{L^\gamma (K)}^\gamma
\right)^{\frac{1}{\gamma}}\\[2mm]
&=:I_1+I_2.
\end{align*}
$I_2$ is estimated similarly as for $I_1$, so we will estimate $I_1$.
Since $(u_m)_+-(u_m)_+^{(\ell)}=0$ in $K \cap \left\{\frac{1}{\ell}<(u_m)_+<\ell\right\}\,\,(m = 1, 2, \ldots)$ from~\eqref{lemma 4.3 eq.3} and~\eqref{lemma 4.3 eq.4} we observe that
\begin{align*}
\left\|(u_m)_+-(u_m)_+^{(\ell)}\right\|_{L^\gamma(K)} 
&\leq \left\|(u_m)_+-(u_m)_+^{(\ell)}\right\|_{L^\gamma(K\,\cap\, \{(u_m)_+ \leq \frac{1}{\ell}\})}+\left\|(u_m)_+-(u_m)_+^{(\ell)}\right\|_{L^\gamma(K\,\cap\, \{(u_m)_+\geq\ell\})} \notag \\[5mm]
&\leq \left\|(u_m)_+\right\|_{L^\gamma(K\,\cap\, \{(u_m)_+ \leq \frac{1}{\ell}\})}+\Big\|\underbrace{(u_m)_+^{(\ell)}}_{\equiv\, \frac{1}{\ell}}\Big\|_{L^\gamma(K\,\cap\, \{(u_m)_+ \leq \frac{1}{\ell}\})} \notag \\[3mm]
&\quad \quad \quad +\left\|(u_m)_+\right\|_{L^\gamma(K\,\cap\, \{(u_m)_+\geq\ell\})}+\Big\|\underbrace{(u_m)_+^{(\ell)}}_{\equiv \,\ell}\Big\|_{L^\gamma(K\,\cap\, \{(u_m)_+\geq\ell\})} \notag \\[3mm]
&\leq \dfrac{1}{\ell} \left|K \cap\left\{(u_m)_+ \leq \frac{1}{\ell}\right\}\right|^{\frac{1}{\gamma}}+\dfrac{1}{\ell} \left|K \cap \left\{(u_m)_+ \leq \frac{1}{\ell} \right\}\right|^{\frac{1}{\gamma}}\notag \\[3mm]
&\quad \quad \quad +\left\|(u_m)_+\right\|_{L^\gamma(K\,\cap\, \{(u_m)_+\geq\ell\})}+\ell \left|K \cap\left\{(u_m)_+ \geq \ell\right\}\right|^{\frac{1}{\gamma}} \notag\\[3mm]
&\!\!\!\!\!\!\!\!\!\!\stackrel{\eqref{lemma 4.3 eq.3},~\eqref{lemma 4.3 eq.4}}{\leq} \dfrac{2}{\ell}\left|K\right|^{\frac{1}{\gamma}}+\frac{C}{\ell^{\frac{p_s^\star}{\gamma}-1}}[u_0]_{W^{s,p}(\bR^n)}^{\frac{p_s^\star}{\gamma}}+\ell \left(\frac{C}{\ell^{p_s^\star}}[u_0]_{W^{s,p}(\bR^n)}^{p_s^\star}\right)^{\frac{1}{\gamma}} \notag \\[2mm]
&=C\left(\frac{1}{\ell}\left|K\right|^{\frac{1}{\gamma}}+\frac{[u_0]_{W^{s,p}(\bR^n)}^{\frac{p_s^\star}{\gamma}}}{\ell^{\frac{p_s^\star}{\gamma}-1}}\right).
\end{align*}
This leads to 
\[
I_1 \leq C(T+1)\left(\frac{1}{\ell^\gamma}\left|K\right|+ \frac{[u_0]_{W^{s,p}(\bR^n)}^{p_s^\star}}{\ell^{p_s^\star-\gamma}}\right),
\]
which in turn gives
\begin{equation}\label{lemma 4.3 eq.5}
\mathbf{I}_{h,\ell} \leq C(T+1)\left(\frac{1}{\ell^\gamma}\left|K\right|+ \frac{[u_0]_{W^{s,p}(\bR^n)}^{p_s^\star}}{\ell^{p_s^\star-\gamma}}\right) =:\zeta(\ell).
\end{equation}
%
Clearly, $\mathbf{I}_{h', \ell}$ is estimated similarly as $\mathbf{I}_{h, \ell}$.
Joining~\eqref{lemma 4.3 eq.5} and~\eqref{lemma 4.3 eq.2} with~Lemma~\ref{Cauchy}, we come up with
\begin{align*}
\limsup_{h,h'\searrow 0}\left\|(u_h)_+-(u_{h'})_+\right\|_{L^\gamma(K_T)} &\leq \limsup_{h\searrow 0} \mathbf{I}_{h,\ell}+\limsup_{h,h'\searrow 0} \mathbf{I}_{h,h',\ell}+\limsup_{h'\searrow 0} \mathbf{I}_{h',\ell} \\[3mm]
&\leq 2\zeta(\ell).
\end{align*}
Noticing by $p_s^\star >\gamma$ that $\lim \limits_{\ell \to \infty}\zeta(\ell)=0$ and sending $\ell \to \infty$ in the above display, we finally arrive at 
\[
\lim\limits_{h,h'\searrow 0}\left\|(u_h)_+-(u_{h'})_+\right\|_{L^\gamma(K_T)}=0.
\]
Similarly as above we also have, 
\[
\lim\limits_{h,h'\searrow 0}\left\|(u_h)_--(u_{h'})_-\right\|_{L^\gamma(K_T)}=0.
\]
\end{proof}
\subsection{Convergence of approximate solutions I}
Here we shall present the convergence of the approximate solutions $u_h$ and $\bar{u}_h$.
\begin{lem}[Convergence of approximate solutions I]\label{convergence result0}
For any $\gamma \in [1, \max\{p_s^\star, q+1\})$, there exists a limit function $u \in L^\gamma (\Omega_T)$ such that, for any a Lipschitz bounded domain $K \subset \bR^n$ and every positive number $T< \infty$,
\begin{align}
&u_h,\,\bar{u}_h\to u \quad  \textrm{a.e.\,\,in}\,\,K_T,
\label{conv. 3} \\
&u_h,\,\bar{u}_h \to u \quad \textrm{strongly\,\,in}\,\,L^\gamma(K_T). \label{conv. 2}
\end{align}

\end{lem}
\begin{proof}
Noticing that $u_h=(u_h)_+-(u_h)_-$, by Lemma~\ref{Cauchyuh}, for all $\gamma$, $1 \leq \gamma<\min\{p_s^\star,\frac{n+1}{n+1-s^\prime}\}$, there exists a limit function $u \in L^\gamma(K_T)$ such that
\begin{equation}\label{x}
u_h \to u \quad \mathrm{strongly\,\,in}\,\, L^\gamma(K_T)
\end{equation}
as $h \searrow 0$. Thus, we have that a subsequence $\{ u_h \}_{h > 0}$ denoted by the same notation, converges to $u$ almost everywhere in $K_T$, that is the first statement of~\eqref{conv. 3}. This fact is combined with~\eqref{energy est. NE1'} in Lemma~\ref{energy est. NE'}, ~\eqref{energyumeq.1}$_2$ in Lemma~\ref{energyum} to yield through Fatou's lemma that, for $\gamma=p_s^\star$ and $\gamma=q+1$,
\begin{equation}\label{alpha}
\sup_{0 < t < T}\|u(t)\|_{L^\gamma (K)}\leq C \max \Big\{ \|u_0\|_{L^{q+1} (\Omega)},
\,[u_0]_{W^{s, p} (\bR^n)}\Big\}.
\end{equation}
For the validity of second statement of~\eqref{conv. 3}, we employ Lemma~\ref{convergence lemma a} and the truncated argument, as discussed in Section~\ref{Sect. Convergence results for the truncated function}. 
%
%
For this, we shall handle $\{(\bar{u}_h)_+\}_{h>0}$ only, as the other case is similarly treated. By the Minkowski inequality we have
\begin{align}\label{lemma 4.2 eq.4}
&\left\|(\bar{u}_h)_+-(u_h)_+ \right\|_{L^1(K_T)} \notag\\[3mm]
&\leq \left\|(\bar{u}_h)_+-(\bar{u}_h)_+^{(\ell)} \right\|_{L^1(K_T)}+\left\|(\bar{u}_h)_+^{(\ell)}-(u_h)_+^{(\ell)} \right\|_{L^1(K_T)}+\left\|(u_h)_+^{(\ell)}-(u_h)_+ \right\|_{L^1(K_T)} \notag\\[3mm]
&=:\mathbf{I}_{h,\ell}+\mathbf{II}_{h,\ell}+\mathbf{III}_{h,\ell}.
\end{align}
Arguing similarly as~\eqref{lemma 4.3 eq.5}, we infer that
\begin{equation}\label{lemma 4.2 eq.5}
\mathbf{I}_{h,\ell}, \,\,\mathbf{III}_{h,\ell} \leq C(T+1)\left(\frac{1}{\ell}\left|K\right|+ \frac{[u_0]_{W^{s,p}(\bR^n)}^{p_s^\star}}{\ell^{p_s^\star-1}}\right) =:\zeta(\ell),
\end{equation}
where $\zeta(\ell) \to 0$ as $\ell \to \infty$.
Similarly as in~\eqref{2} in Lemma~\ref{elementary est.}, we have
\begin{align*}
\left|(\bar{u}_h)_+^{(\ell)}-(u_h)_+^{(\ell)} \right| \leq h \left|\partial_t (u_h)_+^{(\ell)} \right|
\end{align*}
and thus, 
\[
\mathbf{II}_{h,\ell}=\left\|(\bar{u}_h)_+^{(\ell)}-(u_h)_+^{(\ell)} \right\|_{L^1(K_T)} \leq h \left\|\partial_t (u_h)_+^{(\ell)} \right\|_{L^1(K_T)}.
\]
Noting that $q+1>1$, from~\eqref{truncate energy1} and ~\eqref{truncate energy2} in Lemma~\ref{truncate energy} we find
\begin{equation}\label{lemma 4.2 eq.6}
\lim_{h \searrow 0}\mathbf{II}_{h,\ell}=0.
\end{equation}
From~\eqref{lemma 4.2 eq.4},~\eqref{lemma 4.2 eq.5} and~\eqref{lemma 4.2 eq.6}, we obtain
\begin{align*}
\limsup_{h \searrow 0}\left\|(\bar{u}_h)_+-(u_h)_+ \right\|_{L^1(K_T)} &\leq \limsup_{h \searrow 0}\mathbf{I}_{h,\ell}+\limsup_{h \searrow 0}\mathbf{II}_{h}+\limsup_{h \searrow 0}\mathbf{II}_{h} \\[3mm]
&\leq 2\zeta(\ell)
\end{align*}
and, taking the limit as $\ell \to \infty$ we have
\begin{equation}\label{y}
\lim_{h \searrow 0}\left\|(\bar{u}_h)_+-(u_h)_+ \right\|_{L^1(K_T)}=0.
\end{equation}
%
%
%
%
%
%
%
Furthermore, we make the inequality
\[
\|(\bar{u}_h)_+ - u_+\|_{L^1 (K_T)} \leq \|(\bar{u}_h)_+ - (u_h)_+\|_{L^1 (K_T)} + \|(u_h)_+ - u_+\|_{L^1 (K_T)},
\]
which converge to zero as $h \searrow 0$, by the convergences~\eqref{x} and~\eqref{y}.
Thus, a subsequence $\{ (\bar{u}_h)_+ \}_{h>0}$ converges to $u_+$ almost everywhere in $K_T$ as $h \searrow 0$. Analogously, the convergence  $(\bar{u}_h)_- \to u_-$ a.e. in $K_T$ is verified.
Hence the second assertion of~\eqref{conv. 3} is concluded. 

For the proof~\eqref{conv. 2}, we verify that $\left\{(\bar{u}_h)_+ - u_+ \right\}_{h > 0}$ is bounded in $L^\gamma (K_T)$ for $\gamma=p_s^\star$ and $\gamma=q + 1$. In fact, by ~\eqref{energy est. NE1}$_1$ in Lemma~\ref{energy est. NE}, ~\eqref{energyumeq.1}$_1$ in Lemma~\ref{energyum} and~\eqref{alpha}
\[
\|(\bar{u}_h)_+ - u_+\|_{L^\gamma (K_T)} \leq C \max \Big\{\|u_0\|_{L^{q + 1} (\Omega)}, \,[u_0]_{W^{s, p} (\bR^n)}
\Big\}.
\]
Therefore, from Lemma~\ref{convergence lemma a} it follows that, for any $\gamma \in [1,  \max \{p_s^\ast, q+1 \})$, 
\[
\lim_{h \searrow 0} \|({\bar u}_h)_+ - u_+\|_{L^\gamma (K_T)} = 0.
\]
Similarly as above we also have 
\[
\lim_{h \searrow 0}\left\|(\bar{u}_h)_--u_- \right\|_{L^\gamma(K_T)}=0.
\]
Joining two displays above, we actually have the second statement of~\eqref{conv. 2}. The first statement of~\eqref{conv. 2} is shown in the same way as for the second one above. The proof is concluded.
\end{proof}

\subsection{Convergence of approximate solutions II}
 The following convergence results actually holds true:

\begin{lem}[Convergence of approximate solutions II]\label{convergence result}
Let $u_h$,\,$\bar{u}_h$,\,$\bar{v}_h$,\,$v_h$,\,$w_h$ and $\bar{w}_h$ be the approximate solutions of~\eqref{maineq}
defined by~\eqref{approx. sol. for maineq.1}--\eqref{approx. sol. for maineq.2}. For $s \in (0,1)$, let $s^\prime <s$ be a positive number. 
Then there exist subsequences of $\{u_h\}_{h>0}$,\,$\{\bar{u}_h\}_{h>0}$,\,$\{v_h\}_{h>0}$,\,$\{\bar{v}_h\}_{h>0}$,\,$\{w_h\}_{h>0}$ and $\{\bar{w}_h\}_{h>0}$ (denoted by the same symbol unless otherwise stated) and a limit function $u \in L^\infty(0,T\,;L^{q+1}(\Omega))$ such that,
for any bounded domain $K \subset \bR^n$, for every positive number $T<\infty$, the following convergences hold true:
\begin{align}
\bar{u}_h \to u \quad &\ast\textrm{-weakly\,\,in}\,\,L^\infty(0,\infty\,; L^{q+1}(\Omega)),
\label{conv. 1} \\
w_h,  \bar{w}_h \to |u|^{\frac{q - 1}{2}} u
\quad &\textrm{strongly\,\,in}\,\,L^\beta(K_T), \quad \forall \beta \in [1,2),
\label{conv. a} \\
v_h,\,\bar{v}_h \to |u|^{q-1}u \quad &\textrm{strongly\,\,in}\,\,L^\alpha(K_T),\quad \forall \alpha \in [1,\tfrac{q+1}{q})
\label{conv. 4} 
\end{align}
as $h\searrow 0$. 
\end{lem}
\begin{proof}[\bf{Proof of Lemma~\ref{convergence result}}]
To show each of the convergence in the statement, we shall use the energy estimates in Lemmata \ref{energy est. NE} and
\ref{L2-est.}. 
Let us prove Lemma~\ref{convergence result} step by step, as displayed below.

\textbf{Step 1: The proof of~\eqref{conv. 1}.}  
By~\eqref{energy est. NE1} in Lemma \ref{energy est. NE}, we have 
\begin{equation}\label{lemma 4.2 eq.1}
\{\bar{u}_h\}_{h>0}:\quad \textrm{bounded\,\,in}\,\, L^{\infty}(0,\infty ; L^{q+1}(\Omega)).
\end{equation}
Then there exist a subsequence $\{\bar{u}_h\}$ (written by the same symbol) and the limit function $u \in L^\infty (0, \infty ; L^{q + 1} (\Omega))$ such that~\eqref{conv. 1} holds true. Note that this limit function $u$ is equal to that obtained in Lemma ~\ref{convergence result0}, by the convergence in Lemma~\ref{convergence result0} of $u_h$ and $\bar{u}_h$.

\medskip

\textbf{Step 2: The proof of~\eqref{conv. a}.}
We further divide this step into two steps:
\smallskip

\textbf{Step 2.1.}
We shall prove the strong convergence $\bar{w}_h$ to $|u|^{\frac{q-1}{2}}u$  in $L^\beta(K_T)$ for $\beta \in [1, 2)$. By the energy estimate~\eqref{energy est. NE1} in Lemma~\ref{energy est. NE}, we have
\begin{align*}
\iint_{K_T} \big| |\bar{u}_{h}|^{\frac{q-1}{2}}\bar{u}_{h}-|u|^{\frac{q-1}{2}}u\big|^{2}\,dxdt \leq  2\iint_{\Omega_T}\left(|u_0|^{q+1} + |u|^{q+1}\right)\,dx 
\end{align*}
and thus,  $\{ \bar{w}_h-|u|^{\frac{q-1}{2}}u\}_{h > 0}$ is bounded in $L^2 (K_T)$. By ~\eqref{conv. 3} $\bar{w}_h$ converges to $|u|^{\frac{q - 1}{2}} u$ almost everywhere in $K_T$.  From Lemma~\ref{convergence lemma a} with $\psi \equiv 1$ it follows that $\displaystyle \iint_{K_T} \big| |\bar{u}_{h}|^{\frac{q-1}{2}}\bar{u}_{h} -|u|^{\frac{q-1}{2}}u\big|^{\beta}\,dxdt \to 0$, that is,
\begin{equation}\label{lemma 4.2 eq.7}
\bar{w}_h \to |u|^{\frac{q-1}{2}}u\quad  \mathrm{strongly\,\,in}\,\, L^\beta(K_T). 
\end{equation}
This proves the first statement of~\eqref{conv. a}.
\smallskip

\textbf{Step 2.2.} We now observe the difference of $\bar{w}_h$ and $w$. By~\eqref{2} in Lemma~\ref{elementary est.} we have, for $(x,t) \in  \bR^n \times [0,\infty)$, 
\begin{align*}
\left|w_h-\bar{w}_h\right| \leq h\left|\partial_t w_h\right|
\end{align*}
and thus, by the energy estimate~\eqref{(i)} in Lemma~\ref{L2-est.}, 
\begin{align*}
\left\|w_h-\bar{w}_h \right\|_{L^2(K_T)} &\leq h \left\|\partial_t w_h\right\|_{L^2(K_T)} \leq Ch[u_0]_{W^{s,p}(\bR^n)}^{\frac{p}{2}}
\end{align*}
which converges to zero as $h \searrow 0$. This together with~\eqref{lemma 4.2 eq.7} finishes the proof of~\eqref{conv. a}.
\smallskip

\textbf{Step 3: The proof of~\eqref{conv. 4}}. Similarly as in Step 3, we proceed the proof. 
\smallskip

\textbf{Step 3.1.}
Let $\alpha \in [1,\tfrac{q+1}{q})$ be any number. Using the energy estimate~\eqref{energy est. NE1} in Lemma~\ref{energy est. NE}, we have
\begin{align}\label{7}
\iint_{K_T} \left||\bar{u}_h|^{q - 1} \bar{u}_h- |u|^{q - 1} u \right|^{\frac{q+1}{q}}\,dxdt \leq 2^{\frac{1}{q}}\iint_{\Omega_T}\left(|u_0|^{q+1} + |u|^{q+1}\right)\,dx 
\end{align}
and thus,  $\{ \bar{v}_h- |u|^{q - 1} u \}_{h > 0}$ is bounded in $L^{\frac{q+1}{q}} (K_T)$. By ~\eqref{conv. 2} $\bar{v}_h$ converges
to $|u|^{q - 1} u$ almost everywhere in $K_T$.  From Lemma~\ref{convergence lemma a} with $\psi \equiv 1$ we obtain that, for $\alpha \in [1,\frac{q}{q+1})$,
\[
{\bar v}_ h \to|u|^{q - 1} u\quad \textrm{strongly\,\,in}\,\,\, L^\alpha (K_T).
\]

\textbf{Step 3.2.}
Finally, we shall verify the convergence of $v_h$ to $|u|^{q-1}u$ in $L^\alpha(K_T)$ for $\alpha \in [1,\frac{q+1}{q})$. To conclude the proof we now distinguish two cases. Whenever $q\geq 1$, there holds that 
\[
\left| |\bar{u}_h|^{q-1}\bar{u}_h-v_h\right|=\left|\bar{v}_h-v_h\right| \leq h \left|\partial_t v_h \right|
\]
and thus, by the energy estimate~\eqref{(ii)} in Lemma~\ref{L2-est.}
\begin{align}\label{8}
\left\|\bar{v}_h-v_h\right\|_{L^1(K_T)} &\leq h\left\|\partial_t v_h \right\|_{L^1(K_T)} \notag \\
& \leq Ch\left|K_T\right|^{\frac{1}{q+1}}\|u_0\|_{L^{q+1}(\Omega)}^{\frac{q-1}{2}}[u_0]_{W^{s,p}(\bR^n)}^{\frac{p}{2}}.
\end{align}
When $0<q<1$, by~\eqref{3} in Lemma~\ref{elementary est.} and~\eqref{Alg1} in Lemma~\ref{Algs} we estimate for any $(x,t) \in \bR^n \times [0,\infty)$, 
\begin{align*}
|v_h - {\bar v}_h|&\leq C \left(|\bar{u}_h (t)|+ |\bar{u}_h (t - h)|\right)^{q - 1}\left|
\bar{u}_h(t)-\bar{u}_h (t-h)\right| \\[2mm]
&\leq C \left|\bar{u}_h (t) - \bar{u}_h (t - h)\right|^q\\[2mm]
&= C h^q \left(|\bar{u}_h (t)| + |\bar{u}_h (t - h)|\right)^{\frac{q (1- q)}{2}}\left(|\bar{u}_h (t)| + |\bar{u}_h (t - h)|\right)^{\frac{q (q - 1)}{2}}|\partial_t u_h|^q
\end{align*}
and thus, we subsequently use H\"{o}lder's inequality twice with pair of exponents
$\left(\frac{2}{q}, \frac{2}{2 - q}\right)$ and $\left(\frac{(2-q)(q+1)}{q(1-q)},\,\frac{(2-q)(q+1)}{2}\right)$, the energy inequalities~\eqref{energy est. NE1} and~\eqref{energy est. NE3} in Lemma~\ref{energy est. NE} to have 
\begin{align}\label{9}
\iint_{K_T}|v_h - \bar{v}_h|\,dxdt &\leq h^q\left(\,\iint_{K_T}\left(|\bar{u}_h (t)|+ |\bar{u}_h (t - h)|\right)^{q - 1}|\partial_t u_h|^2\,dxdt\right)^{\frac{q}{2}} \notag\\[3mm]
&\quad \quad \quad \quad \quad \times \left(\,\iint_{K_T}\left(|\bar{u}_h (t)|+ |\bar{u}_h (t - h)|\right)^{\frac{q (1 - q)}{2- q}}\,dxdt\right)^{\frac{2 - q}{2}} \notag \\[3mm]
&\leq h^q\left(\,\iint_{K_T}\left(|\bar{u}_h (t)|+ |\bar{u}_h (t - h)|\right)^{q - 1}|\partial_t u_h|^2\,dxdt\right)^{\frac{q}{2}} \notag\\[3mm]
&\quad \quad \quad \quad \quad \times \left(\,\iint_{K_T}\left(|\bar{u}_h (t)|+ |\bar{u}_h (t - h)|\right)^{q+1}\,dxdt\right)^{\frac{q(1-q)}{2(q+1)}}\left|K_T\right|^{\frac{1}{q+1}} \notag \\[3mm]
&\leq
Ch^q  \left|K_T\right|^{\frac{1}{q+1}}
\|u_0\|_{L^{q + 1} (\Omega)}^{\frac{q(1-q)}{2}}
[u_0]_{W^{s, p} (\bR^n)}^{\frac{p q}{2}}.
\end{align}
Passing to the limit $h \searrow 0$ in~\eqref{8} and~\eqref{9} from the convergence in Step 3.2, we deduce the convergence  in $L^1 (K_T)$ of $v_h$ and thus, convergence almost everywhere in $K_T$ of a subsequence of $v_h$ to $|u|^{q - 1} u$. 

Similarly as~\eqref{7} above we also see that $\left\{v_h - |u|^{q-1} u\right\}_{h>0}$ is bounded in $L^{\frac{q + 1}{q}} (K_T)$. Thus, by Lemma~\ref{convergence lemma a}, this together with the almost convergence show the validity of the convergence of $v_h$ to $|u|^{q-1}u$ in $L^\alpha(K_T)$ for $\alpha \in [1,\frac{q+1}{q})$.
\medskip

Finally, the proof of Lemma~\ref{convergence result} is complete.
\end{proof}

From Lemma~\ref{convergence lemma a} we easily deduce the following  convergence, used later in Section~\ref{Sect. 5}.
\begin{lem}\label{convergence result2}
For every test function $\varphi \in \mathscr{T}$, there holds that
\begin{equation}\label{wh}
-\iint_{\Omega_T} w_h \, \partial_t \varphi \,dxdt \to -\iint_{\Omega_T} |u|^{\frac{q-1}{2}}u \,\partial_t \varphi \,dxdt
\end{equation}
and
\begin{equation}\label{vh}
-\iint_{\Omega_T} v_h \, \partial_t \varphi \,dxdt \to -\iint_{\Omega_T} |u|^{q-1}u \,\partial_t \varphi \,dxdt
\end{equation}
as $h \searrow 0$.
\end{lem}
\begin{proof}
We shall show~\eqref{vh} only, as the other case~\eqref{wh} is similarly treated.
%
%
%
By~\eqref{energy est. NE1'} in Lemma~\ref{energy est. NE'}
\begin{align*}
\iint_{\Omega_T}|v_h|^{\frac{q+1}{q}}\,dxdt \leq C(T+1)\|u_0\|_{L^{q+1}}^{q+1}
\end{align*}
and thus, $\{v_h\}_{h>0}$ is bounded in $L^{\frac{q+1}{q}}(\Omega_T)$. From~\eqref{conv. 4} in Lemma~\ref{convergence result}
\[
v_h \, \partial_t \varphi \to |u|^{q-1}u\,\partial_t \varphi \quad \textrm{a.e.\,\,in}\,\,\Omega_T.
\]
%
%
Noticing that $\partial_t \varphi \in L^{q+1}(\Omega_T)$ since $\varphi \in \mathscr{T}$ and choosing $\psi=\partial_t\varphi$ in Lemma~\ref{convergence lemma a} lead to
\[
\iint_{\Omega_T} v_h \, \partial_t \varphi \,dxdt \to \iint_{\Omega_T} |u|^{q-1}u \, \partial_t \varphi \,dxdt
\]
as $h \searrow 0$, finishing the proof.
\end{proof}

%
%

\section{Proof of Theorem~\ref{mainthm}}\label{Sect. 5}

Now we are in the position to prove the main Theorem~\ref{mainthm}  since we have all prerequisites at hand. 
\begin{proof}[\normalfont \textbf{Proof of Theorem~\ref{mainthm}}]
We shall show that the limit function $u$ obtained in Lemma~\ref{convergence result} is a weak solution of ~\eqref{NE}. For this we will verify that $u$ obeys the conditions in Definition~\ref{def. of weak sol.}. Firstly, we ensure the condition (D1) in Definition~\ref{def. of weak sol.}. 
%
%
By the convergence~\eqref{conv. 3} in Lemma~\ref{convergence result}, applied to~\eqref{energy est. NE4} in Lemma~\ref{energy est. NE} and Fatou's lemma we have that, for any positive $t<\infty$,
\begin{equation}\label{Proof eq.7}
[u(t)]_{W^{s,p}(\bR^n)} \leq \liminf_{h \searrow 0}\,[\bar{u}_h(t)]_{W^{s,p}(\bR^n)} \leq [u_0]_{W^{s,p}(\bR^n)}.
\end{equation}
Furthermore, by~\eqref{e.Poincare} in Lemma~\ref{t.Poincare} and~\eqref{Proof eq.7}
we have, for any positive $t < \infty$,
\begin{equation}\label{Proof eq.a}
\|u (t)\|_{L^p (\Omega)}\leq C [u_0]_{W^{s, p} (\bR^n)}.
\end{equation}
Thus, it follows from~\eqref{Proof eq.7} and~\eqref{Proof eq.a} that the limit function $u$ is in $L^\infty (0, \infty ; W^{s, p}_0 (\Omega))$. Similarly as~\eqref{Proof eq.7}, by~\eqref{energy est. NE1'} in Lemma~\ref{energy est. NE'}, we also have, for any positive $t < \infty$,
\begin{equation}\label{Proof eq.b}
\|u (t)\|_{L^{q + 1} (\Omega)} \leq \|u_0\|_{L^{q + 1} (\Omega)}
\end{equation}
which implies $u \in L^\infty(0,\infty ; L^{q+1}(\Omega))$ and thus, the proof of (D1) is accomplished.
%
%

By~\eqref{(i)} in Lemma~\ref{L2-est.}, $\{\partial_tw_h\}_{h>0}$ is bounded in $L^2(\Omega_\infty)$ and thus, there are a subsequence $\{\partial_tw_h\}_{h>0}$ (denoted by the same notation)
and a limit function $\omega \in L^2(\Omega_\infty)$ such that, as $h \searrow 0$,
\begin{equation}\label{Proof eq.1}
\partial_t w_h \to \omega \quad \textrm{weakly\,\,in}\,\,\,L^2(\Omega_\infty).\end{equation}
By use of the convergence~\eqref{wh} in Lemma~\ref{convergence result2} we pass to the limit as $h \searrow 0$ in the identity
\begin{equation}\label{Proof eq.2}
\iint_{\Omega_\infty}\partial_t w_h\cdot \varphi\,dxdt =-\iint_{\Omega_\infty}w_h\,\partial_t \varphi\,dxdt
\end{equation}
for every $\varphi \in C^\infty_0(\Omega_\infty)$ to have
\begin{equation}\label{Proof eq.3}
\omega=\partial_t(|u|^{\frac{q-1}{2}}u)\quad \textrm{in}\,\,L^2(\Omega_\infty).
\end{equation}
As a result, it follows from~\eqref{Proof eq.1} and~\eqref{Proof eq.3} that
\begin{equation}\label{Proof eq.A}
\partial_t w_h \to \partial_t(|u|^{\frac{q-1}{2}}u) \quad \textrm{weakly\,\,in}\,\,\,L^2(\Omega_\infty)
\end{equation}
and thus, from the convergence~\eqref{Proof eq.A} and~\eqref{(i)} in Lemma~\ref{L2-est.}, for any $T<\infty$,
\begin{equation}\label{Proof eq.8}
\left\|\partial_t (|u|^{\frac{q-1}{2}}u)\right\|_{L^2(\Omega_T)}^2 \leq C[u_0]_{W^{s,p}(\bR^n)}^p.
\end{equation}
Therefore the regularity part of Theorem~\ref{mainthm} is concluded.
%
%

We now turn to prove the condition (D2) in Definition~\ref{def. of weak sol.}.
Let $\varphi \in \mathscr{T}$, whose the definition is given in Definition~\ref{def. of weak sol.}. According to the notation in \eqref{Utxy}, we first set
\begin{align*}
\overline{U}_h(x,y,t)
&:=|\bar{u}_h(x,t)-\bar{u}_h(y,t)|^{p-2}(\bar{u}_h(x,t)-\bar{u}_h(y,t)),\\[2mm]
U(x,y,t)&:=|u(x,t)-u(y,t)|^{p-2}(u(x,t)-u(y,t)).
\end{align*}
For the estimation below let $R$ be any positive number such that the open ball $B_{R}$ with center of origin and radius $R$ compactly contains $\Omega$. In order to verify (D2) in Definition~\ref{def. of weak sol.}, we shall show the convergence, as $h\searrow 0$,
\begin{equation}\label{Proof eq.4}
\int_0^T \iint_{\bR^n \times \bR^n}
\dfrac{\overline{U}_h(x,y,t)\left(\varphi(x,t)-\varphi(y,t)\right)}{|x-y|^{n+sp}}\,dxdydt  \to
\int_0^T \iint_{\bR^n \times \bR^n}
\dfrac{U(x,y,t)\left(\varphi(x,t)-\varphi(y,t)\right)}{|x-y|^{n+sp}}\,dxdydt.
\end{equation}
%
%

We now set $B \equiv B_{R}$, $B^c:=\bR^n \setminus B$ and $D_{B}:=(\bR^n \times \bR^n ) \setminus (B^c \times B^c)$. 
For the validity of~\eqref{Proof eq.4}, we shall consider the difference of two fractional integrations on $D_B$, subsequently, we divide it into three terms as follows:
\begin{align}\label{Proof eq.5}
I_h&:=\left|\int_0^T\iint_{D_{B}} \Bigg[
\dfrac{\overline{U}_h(x,y,t)\left(\varphi(x,t)-\varphi(y,t)\right)}{|x-y|^{n+sp}} - \dfrac{U(x,y,t)\left(\varphi(x,t)-\varphi(y,t)\right)}{|x-y|^{n+sp}}\Bigg]\,dxdydt \right| \notag\\[2mm]
&\leq\int_0^T \iint_{B \times B}\ldots\,dxdydt+\int_0^T\iint_{B\times B^c}\ldots\,dxdydt+\int_0^T\iint_{B^c \times B}\ldots\,dxdydt
\notag\\[2mm]
&=:I_h^{(1)}+I_h^{(2)}+I_h^{(3)}.
\end{align}
%
%
%
%

We now proceed estimating the terms $I_h^{(1)},I_h^{(2)}$ and $I_h^{(3)}$. For the estimation of $I_h^{(1)}$,
we shall verify that
$\left\{\dfrac{\overline{U}_h(x,y,t)\left(\varphi(x,t)-\varphi(y,t)\right)}{|x-y|^{n+sp}}\right\}_{h>0}$
is uniformly integrable on $B \times B \times (0,T)$. 
Let $\Lambda \subset B \times B \times (0,T)$ be any measurable set.
By H\"{o}lder's inequality
and energy estimate~\eqref{energy est. NE2} in Lemma~\ref{energy est. NE},
we have
\begin{align*}
&\iiint_{\Lambda}\left|\dfrac{\overline{U}_h(x,y,t)\left(\varphi(x,t)-\varphi(y,t)\right)}{|x-y|^{n+sp}}\,\right|\,dxdydt \notag\\
&\leq \left(\iiint_{\Lambda}\dfrac{|\varphi(x,t)-\varphi(y,t)|^p}{|x-y|^{n+sp}}\,dxdydt \right)^\frac{1}{p} \left(\int_0^T\iint_{\bR^n \times \bR^n}\dfrac{|\bar{u}_h(x,t)-\bar{u}_h(y,t)|^p}{|x-y|^{n+sp}}\,dxdydt \right)^\frac{p-1}{p} \notag\\
&\!\!\stackrel{\eqref{energy est. NE2}}{\leq}C\left(\iiint_{\Lambda}\dfrac{|\varphi(x,t)-\varphi(y,t)|^p}{|x-y|^{n+sp}}\,dxdydt \right)^\frac{1}{p} \left(\int_\Omega |u_0|^{q+1}\,dx \right)^\frac{p-1}{p}.
\end{align*}
Due to the absolute continuity of the first Lebesgue integral on the right-hand side, for every positive number $\varepsilon$,
there exists a positive number $\delta$ such that
\[
\iiint_{\Lambda}
\left|\dfrac{\overline{U}_h(x,y,t)\left(\varphi(x,t)-\varphi(y,t)\right)}{|x-y|^{n+sp}}\right|\,dxdydt
<\varepsilon\quad \textrm{whenever}\quad
|\Lambda| <\delta
\]
and thus, $\left\{\dfrac{\overline{U}_h(x,y,t)\left(\varphi(x,t)-\varphi(y,t)\right)}{|x-y|^{n+sp}}\right\}_{h>0}$
is uniformly integrable on $B \times B \times (0,T)$.
From the almost everywhere convergence~\eqref{conv. 3} in Lemma~\ref{convergence result},
it follows that 
\begin{align*}
\dfrac{\overline{U}_h(x,y,t)\left(\varphi(x,t)-\varphi(y,t)\right)}{|x-y|^{n+sp}} \to  \dfrac{U(x,y,t)\left(\varphi(x,t)-\varphi(y,t)\right)}{|x-y|^{n+sp}}
\end{align*}
almost everywhere $(x,y,t) \in B \times B \times (0,T)$, 
where we notice that the set
\[
\left\{(x,y,t)\in B \times B \times (0,T): x=y \right\}
\]
is negligible with respect to the $(2n+1)$-dimensional Lebesgue measure on $\bR^n \times \bR^n \times \bR$.
Therefore, Vitalli's convergence theorem yields that
\begin{align*}
\int_0^T\iint_{B \times B}\dfrac{\overline{U}_h(x,y,t)\left(\varphi(x,t)-\varphi(y,t)\right)}{|x-y|^{n+sp}}\,dxdydt \to  \int_0^T\iint_{B \times B}\dfrac{U(x,y,t)\left(\varphi(x,t)-\varphi(y,t)\right)}{|x-y|^{n+sp}}\,dxdydt,
\end{align*}
as $h\searrow 0$, that is,
\begin{equation}\label{I_1}
\lim_{h\searrow 0} I_h^{(1)} =0.
\end{equation}
We are going to estimate $I_h^{(2)}$ and $I_h^{(3)}$.
Since these two integrals are symmetric with respect to $x$ and $y$, $I_h^{(3)}$ can be estimated similarly to $I_h^{(2)}$ and thus, we give an estimate of $I_h^{(2)}$. 
The boundary condition \eqref{NE}$_2$ implies that $\bar{u}_h=u=0$ in $\Omega^c\times (0,T)$ and, by $B^c \subset \Omega^c$, $\bar{u}_h(y,t)=u(y,t)=\varphi(y,t)=0$ for almost all $(y,t) \in B^c \times (0,T)$ because $\Omega$ is compactly contained in $B$ and $u_h = u = 0$ almost everywhere in $\Omega^c \times (0, T)$. Using H\"older's inequality, $I^{(2)}_h$ is estimated as
\begin{align}\label{eq.2 I2}
I_h^{(2)} &\leq \int_0^T \iint_{\Omega \times B^c}
\dfrac{\Big||\bar{u}_h(x,t)|^{p-2}\bar{u}_h(x,t)-|u(x,t)|^{p-2}u(x,t)\Big|}{|x-y|^{n+sp}}|\varphi(x,t)|\,dxdydt\notag \\[3mm]
&\leq \int_0^T \int_{\Omega}\left(\,\,\int_{B^c}\dfrac{dy}{|x-y|^{n+sp}}\right)\Big||\bar{u}_h(x,t)|^{p-2}\bar{u}_h(x,t)-|u(x,t)|^{p-2}u(x,t)\Big||\varphi(x,t)|\,dxdt   \notag\\[3mm]
&\leq C(n,s,p)d^{-sp}\iint_{\Omega_T}\Big||\bar{u}_h(x,t)|^{p-2}\bar{u}_h(x,t)-|u(x,t)|^{p-2}u(x,t)\Big||\varphi(x,t)|\,dxdt,
\end{align}
where, in the second line, we note that  $|x - y| \geq d : = \mathrm{dist} (\Omega, \partial B)$ for any $(x, y) \in \Omega \times B^c$, and estimate for $(x, y) \in \Omega \times B^c$
\begin{align*}
\int_{B^c}\dfrac{dy}{|x-y|^{n+sp}} &\leq \int_{B_d(x)^c}\dfrac{dy}{|x-y|^{n+sp}} \\[3mm]
&=C(n)\int_d^\infty \dfrac{d\rho}{\rho^{1+sp}}=C(n)\dfrac{d^{-sp}}{sp}.
\end{align*}
To further estimate the integral of the right hand side of~\eqref{eq.2 I2} we shall verify that  $\left\{\Big||\bar{u}_h|^{p-2}\bar{u}_h-|u|^{p-2}u\Big|\right\}_{h>0}$ is bounded in  $L^{\frac{p}{p-1}}(\Omega_T)$. By~\eqref{energyumeq.0} in Lemma~\ref{energyum}%
\begin{equation}\label{eq.a I2}
\iint_{\Omega_T} \left|\bar{u}_h(x,t)\right|^p\,dxdt \leq C\int_0^T [\bar{u}_h(t)]_{W^{s, p} (\bR^n)}^p\,dt
\leq C \|u_0\|_{L^{q+1} (\Omega)}^{q+1}.
\end{equation}
By the almost convergence~\eqref{conv. 3} in Lemma~\ref{convergence result},~\eqref{eq.a I2} and Fatou's lemma, we have
\[
\iint_{\Omega_T} \left|u(x,t)\right|^p\,dxdt \leq \liminf_{h \searrow 0}\iint_{\Omega_T} \left|\bar{u}_h(x,t)\right|^p\,dxdt \leq C\|u_0\|_{L^{q+1} (\Omega)}^{q+1}.
\]
This together with~\eqref{eq.a I2} yields%
\begin{align*}
&\iint_{K_T}\Big||\bar{u}_h(x,t)|^{p-2}\bar{u}_h(x,t)-|u(x,t)|^{p-2}u(x,t)\Big|^{\frac{p}{p-1}}\,dxdt \\[3mm]
& \leq 2^{\frac{1}{p-1}}\iint_{\Omega_T} \left(|\bar{u}_h(x,t)|^p+|u(x,t)|^p\right)\,dxdt  \\[3mm]
&\leq C\|u_0\|_{L^{q+1}(\Omega)}^{q+1}
\end{align*}
and thus, $\left\{\Big||\bar{u}_h|^{p-2}\bar{u}_h-|u|^{p-2}u\Big|\right\}_{h>0}$ is bounded in  $L^{\frac{p}{p-1}}(\Omega_T)$. 
This,~\eqref{eq.2 I2} and~\eqref{conv. 3} in Lemma~\ref{convergence result} through Lemma~\ref{convergence lemma a} yields
\begin{equation}\label{eq.5 I2}
\lim_{h \searrow 0}I_h^{(2)} =0.
\end{equation}
The preceding estimates~\eqref{I_1} and~\eqref{eq.5 I2} can be merged to give the following final estimate for $I_h$:
\begin{align*}
\limsup_{h \searrow 0}I_h \leq \limsup_{h \searrow 0}I_h^{(1)}+2\limsup_{h \searrow 0}I^{(2)}_h 
&=0,
\end{align*}
Therefore, by~\eqref{Proof eq.4} and~\eqref{vh} in Lemma~\ref{convergence result2}, we pass to the limit as $h \searrow 0$ in~\eqref{weak form NE'} to have
\[
-\iint_{\Omega_T}|u|^{q-1}u\,\partial_t\varphi\,dxdt+\dfrac{1}{2}\int_0^T\iint_{\bR^n \times \bR^n}\dfrac{U(x,y,t)}{|x-y|^{n+sp}}\left(\varphi(x,t)-\varphi(y,t)\right)\,dxdydt=0
\]
for all $\varphi \in \mathscr{T}$, which conclude the condition (D2). Here we remark that the class $\mathscr{T}$ of test functions in Definition~\ref{def. of weak sol.} is a subspace of that in~\eqref{weak form NE'}.
\medskip

Finally, we shall verify the condition (D3). As noted before, we have that $\bar{u}_h(t)=0$ almost everywhere in $\bR^n \setminus \Omega$ for any $t \geq 0$. Thus, by the convergence~\eqref{conv. 3} in Lemma~\ref{convergence result}, $u=0$ almost everywhere in $(\bR^n \setminus \Omega) \times (0,T)$ for any $T>0$, showing the validity of boundary condition in (D3).
%
%
\medskip

We test the approximating equation~\eqref{weak form NE'} on $\bR^n \times (t_1, t_2)$ for any $t_2 > t_1 \geq 0$,  with $\varphi(x)\sigma_{(t_1,t_2)}(t)$, where $\varphi \in C^\infty_0(\Omega)$ and $\sigma_{(t_1,t_2)}(t)$ is the usual Lipschitz approximation of characteristic function of time-interval $(t_1, t_2)$. By use of the H\"{o}lder inequality and ~\eqref{energy est. NE4} in Lemma~\ref{energy est. NE}, we have that, for any $t_2>t_1 \geq 0$,
\begin{align}\label{Proof eq.10}
\left|\,\,\int_{\Omega}\big(v_h(t_2)-v_h(t_1)\big) \varphi\,dx \right| &=\left|-\frac{1}{2}\int_{t_1}^{t_2}\iint_{\bR^n \times \bR^n}\frac{\overline{U}_h(x,y,t)}{|x-y|^{n+sp}}\big(\varphi(x)-\varphi(y)\big)\,dxdydt \right|\notag\\[3mm]
&\leq \frac{1}{2}\int_{t_1}^{t_2}[\bar{u}_h(t)]_{W^{s,p}(\bR^n)}^{p-1}\,dt\cdot [\varphi]_{W^{s,p}(\bR^n)} \notag \\[3mm]
&\!\!\stackrel{\eqref{energy est. NE4}}{\leq} \frac{1}{2}[u_0]_{W^{s,p}(\bR^n)}^{p-1}[\varphi]_{W^{s,p}(\bR^n)}(t_2-t_1),
\end{align}
implying by the density of $C^\infty_0 (\Omega)$ in $L^{q+1} (\Omega)$ that $v_h (t)$ is weakly continuous on $t \in [0, \infty)$
in $L^{\frac{q+1}{q}} (\Omega)$, uniformly on approximation parameter $h$. By~\eqref{3} in Lemma~\ref{elementary est.} and~\eqref{energy est. NE1} in Lemma~\ref{energy est. NE}, $\{v_h (t)\}_{h > 0}$ is bounded in $L^{\frac{q+1}{q}} (\Omega)$.
Thus, the Arzel\`{a}-Ascoli theorem gives that, as $h \searrow 0$,
\[
v_h \to |u|^{q-1}u 
\]
uniformly locally on $[0, \infty)$ and weakly in $L^{\frac{q+1}{q}} (\Omega)$. This together with~\eqref{Proof eq.4} applied to the first line of~\eqref{Proof eq.10} implies that, for $t_2> t_1= 0$,
\[
\left|\,\,\int_{\Omega}\big(|u|^{q-1}u(t_2)-|u_0|^{q-1}u_0\big) \varphi\,dx \right| =\left|-\frac{1}{2}\int_{0}^{t_2}\iint_{\bR^n \times \bR^n}\frac{U(x,y,t)}{|x-y|^{n+sp}}\big(\varphi(x)-\varphi(y)\big)\,dxdydt \right|
\]
yielding that $|u (t)|^{q - 1} u(t)$ converges to $|u_0|^{q-1}u_0$ weakly in $L^{\frac{q+1}{q}} (\Omega)$, as $t \searrow 0$. Similarly as above, we use the algebraic inequality~\eqref{Alg1} in Lemma~\ref{Algs} to get
\begin{equation}\label{Proof eq.c}
u(t) \to u_0 \quad \textrm{weakly\,\,in}\quad L^{q+1}(\Omega)
\end{equation}
as $t \searrow 0$. Furthermore, the inequality~\eqref{Proof eq.b} gives
\begin{equation}\label{Proof eq.d}
\limsup_{t \searrow 0}
\|u (t)\|_{L^{q + 1} (\Omega)}
\leq \|u_0\|_{L^{q + 1} (\Omega)}.
\end{equation}
From~\eqref{Proof eq.c} and~\eqref{Proof eq.d} it follows that
\begin{equation}\label{Proof eq.11}
u(t) \to u_0\quad \textrm{strongly\,\,in}\,\,L^{q+1}(\Omega)\,\,\,\textrm{as}\,\,t \searrow 0.
\end{equation}

We now have prerequisites at hand to verify the initial condition (D3). In view of the convergence~\eqref{Proof eq.11}, we have only to show $[u(t)-u_0]_{W^{s,p}(\bR^n)} \to 0$  as $t \searrow 0$. We first observe from~\eqref{Proof eq.11} and Fatou's lemma that
\[
[u_0]_{W^{s,p}(\bR^n)} \leq \liminf_{t \searrow 0} [u(t)]_{W^{s,p}(\bR^n)},
\]
which, together with~\eqref{Proof eq.7} yields 
\begin{equation}\label{Proof eq.12}
\lim_{t\searrow 0}[u(t)]_{W^{s,p}(\bR^n)} = [u_0]_{W^{s,p}(\bR^n)}.
\end{equation}
%
%
From the convergence~\eqref{Proof eq.11} and the boundedness~\eqref{Proof eq.7} we can apply Vitali's convergence theorem to have
%
%
%
\[
\dfrac{u(x,t)-u(y,t)}{|x-y|^{s+\frac{n}{p}}} \to \dfrac{u_0(x)-u_0(y)}{|x-y|^{s+\frac{n}{p}}} \quad \textrm{strongly\,\,in}
\,\,L^p(B_R\times B_R)
\]
for any ball $B_R$. 
We shall consider the integral on $(x, y) \in \bR^n \times \bR^n$ with a test function of the difference of two terms above, which corresponds to~\eqref{Proof eq.5} for~\eqref{Proof eq.4}. Arguing similarly as the proof of~\eqref{Proof eq.5} we obtain that

\begin{equation}\label{Proof eq.13}
\dfrac{u(x,t)-u(y,t)}{|x-y|^{s+\frac{n}{p}}} \to  \dfrac{u_0(x)-u_0(y)}{|x-y|^{s+\frac{n}{p}}}\quad \textrm{weakly\,\,in}
\,\,L^p(\bR^n \times \bR^n).
\end{equation}
Gluing~\eqref{Proof eq.12} with~\eqref{Proof eq.13} and adopting the Radon-Riesz theorem, we are led to the conclusion $[u(t)-u_0]_{W^{s,p}(\bR^n)} \to 0$ as $t \searrow 0$ and thus, the initial condition (D3) is actually verified.
%

%
Therefore the proof is complete.
\end{proof}


\makeatletter
\renewcommand{\thesection}{\Alph{section}.\arabic{subsection}}
\makeatother

\appendix

\makeatletter
\renewcommand{\thefigure}{\Alph{section}.\arabic{figure}}
\@addtoreset{figure}{section}
\makeatother

\section{Proof of Lemma~\ref{existence of NE}}\label{Appendix A}
In this appendix, we prove the existence of solutions of~\eqref{NE}, Lemma~\ref{existence of NE}. 
The proof is based on the direct method in the calculus of variations. 
In the course of this appendix, let $s \in (0,1)$, $p>1$ and $q>0$  be given and set $\mathcal{X}:=W_0^{s,p}(\Omega) \cap L^{q+1}(\Omega)$.
%
%
%
%
\begin{proof}[\normalfont \textbf{Proof of Lemma~\ref{existence of NE}}]
Given $u_{m-1} \in \mathcal{X}$, we define the functional on $\mathcal{X}$: 
\[
\mathcal{X} \ni w \longmapsto \mathcal{F}(w):=\dfrac{1}{h}\int_\Omega \left(\dfrac{1}{q+1}|w|^{q+1}-|u_{m-1}|^{q-1}u_{m-1}w \right)\,dx+\dfrac{1}{2p}[w]_{W^{s,p}(\bR^n)}^p.
\]
The proof is twofold: first, we shall show the existence of a minimizer of the functional $\mathcal{F}(w)$, then we shall demonstrate the first variation of $\mathcal{F}(w)$, both of them show the validity of Lemma~\ref{existence of NE}.

By Young's inequality, there holds for any $w \in \mathcal{X}$
\begin{align}\label{new eq.1 A1}
\mathcal{F}(w)&=\dfrac{1}{h}\int_\Omega \left(\dfrac{1}{q+1}|w|^{q+1}-|u_{m-1}|^{q-1}u_{m-1}w \right)\,dx+\dfrac{1}{2p}[w]_{W^{s,p}(\bR^n)}^p \notag\\[3mm]
&\geq \dfrac{1}{h} \int_\Omega \left(\dfrac{1}{2(q+1)}|w|^{q+1}-C(q)|u_{m-1}|^{q+1} \right)\,dx+\dfrac{1}{2p}[w]_{W^{s,p}(\bR^n)}^p \\[3mm]
&>-\infty. \notag
\end{align}
Let $\{u_k\}_{k=1}^\infty$ be a minimizing sequence, such that
\[
\inf \mathcal{F} \leq \mathcal{F} (u_k) < \inf \mathcal{F}+\dfrac{1}{k}
\]
for all $k \in \mathbb{N}$, where the infimum is taken for $w \in \mathcal{X} : = W^{s, p}_0 (\Omega) \cap L^{q + 1} (\Omega)$. Note that, for all $w \in \mathcal{X}$,
\[
- C \|u_{m - 1}\|_{L^{q + 1} (\Omega)}^{q + 1} \leq \inf\mathcal{F} \leq \mathcal{F} (v)  < \infty
\]
with some $v \in \mathcal{X}$ not empty, because $C^\infty_0 (\Omega) \subset \mathcal{X}$. We then observe that $\sup \limits_{k \in \mathbb{N}} \mathcal{F}(u_k) \leq \inf \mathcal{F}+1=:M<\infty$ and, by~\eqref{new eq.1 A1},   
\begin{equation}\label{new eq.2 A1}
\dfrac{1}{h} \int_\Omega \dfrac{1}{2(q+1)}|u_k|^{q+1}\,dx+\dfrac{1}{2p}[u_k]_{W^{s,p}(\bR^n)}^p \leq M+C\int_\Omega |u_{m-1}|^{q+1}\,dx,
\end{equation}
implying
\begin{equation}\label{new eq.2add A1}
\sup \limits_{k \in \mathbb{N}} \|u_k\|_{L^{q+1}(\Omega)}<\infty \quad ; \quad\sup \limits_{k \in \mathbb{N}}\,[u_k]_{W^{s, p} (\bR^n)}<\infty.
\end{equation}
Up to not relabelled subsequences, there exists a weak limit $u \in L^{q+1}(\Omega)$ such that
\begin{equation}\label{new eq.2'' A1}
u_k \to u \quad \textrm{weakly\,\,in}\quad L^{q+1}(\Omega)
\end{equation}
and thus, by the Banach-Steinhaus theorem and~\eqref{new eq.2add A1},
\begin{equation}\label{new eq.2' A1}
\|u\|_{L^{q+1}(\Omega)} \leq \liminf_{k \to \infty} \|u_k\|_{L^{q+1}(\Omega)}<\infty.
\end{equation}
%
%
%

Since $u_k \in W^{s,p}_0(\Omega)$, by the fractional Poincar\'{e} inequality~\eqref{e.Poincare} in Lemma~\ref{t.Poincare} we have, $\sup \limits_{k \in \mathbb{N}}\|u_k\|_{L^p(\bR^n)}<\infty$, in turn yielding with~\eqref{new eq.2add A1}
\begin{equation}\label{new eq.3add A1}
\sup \limits_{k \in \mathbb{N}}\|u_k\|_{W^{s,p}(\bR^n)}<\infty.
\end{equation}
Letting $s^\prime \equiv \min\{\frac{n}{2p},s\}$ impiles $0<s^\prime \leq s <1$. By Lemma~\ref{monotone},
\begin{equation*}
\sup_{k \in \mathbb{N}}\|u_k\|_{W^{s^\prime,p}(\bR^n)}<\infty.
\end{equation*}
and, for any ball $B_R$ satisfying $\Omega \subset B_R$
\begin{equation*}
\sup_{k \in \mathbb{N}}\|u_k\|_{W^{s^\prime,p}(B_R)}<\infty.
\end{equation*}
Since $s^\prime p \leq \frac{n}{2} <n$, we can apply Lemma~\ref{compact embedding} to get (after passing to a subsequence, if necessary)
\[
u_k \to u\quad \textrm{strongly\,\,in}\quad L^r(B_R)
\]
for all $r$ with $1\leq r<\frac{np}{n-s^\prime p}$ and 
\begin{equation}\label{new eq.3 A1}
u_k \to u \quad \textrm{a.e.\,\, in}\quad \Omega,
\end{equation}
which, in particular, yields $u=0$ in $\Omega^c$. From Fatou's Lemma with~\eqref{new eq.3 A1} and~\eqref{new eq.3add A1}, it follows that
\begin{equation}\label{new eq.4 A1}
\|u\|^p_{W^{s,p}(\bR^n)} \leq \liminf_{k \to \infty}\|u_k\|^p_{W^{s,p}(\bR^n)}<\infty.
\end{equation}
%
Consequently, by~\eqref{new eq.2' A1} and~\eqref{new eq.4 A1}, it verifies that $u \in W^{s,p}_0(\Omega) \cap L^{q+1}(\Omega)=\mathcal{X}$. 

Finally, in view of~\eqref{new eq.2' A1} and~\eqref{new eq.4 A1} again, we have
\begin{align*}
\inf_{w \in \mathcal{X}} \mathcal{F}(w)&=\liminf_{k \to \infty}\mathcal{F}(u_k) \\[3mm]
&= \liminf_{k \to \infty}\Bigg[\dfrac{1}{h}\int_\Omega \left(\dfrac{1}{q+1}|u_k|^{q+1}-|u_{m-1}|^{q-1}u_{m-1}u_k \right)\,dx+\dfrac{1}{2p}[u_k]^p_{W^{s,p}(\bR^n)}\Bigg] \\[3mm]
&\geq \dfrac{1}{h}\int_\Omega \left(\dfrac{1}{q+1}|u|^{q+1}-|u_{m-1}|^{q-1}u_{m-1}u \right)\,dx+\dfrac{1}{2p}[u]^p_{W^{s,p}(\bR^n)} \\[3mm]
&=\mathcal{F}(u), 
\end{align*}
that implies that $u \in \mathcal{X}$ is actually a minimizer of $\mathcal{F}$.

We now turn to compute the first variation of $\mathcal{F}$. Let $u_m \in \mathcal{F}$ be a minimizer. For $\varepsilon \in \bR$ and $\varphi \in \mathcal{X}$, we compute as
\begin{align}\label{new eq.a A1}
\dfrac{\mathcal{F}(u_m+\varepsilon \varphi)-\mathcal{F}(u_m)}{\varepsilon}&=\dfrac{1}{h}\dfrac{1}{q+1} \int_\Omega \dfrac{|u_m+\varepsilon\varphi|^{q+1}-|u_m|^{q+1}}{\varepsilon}\,dx \notag \\[3mm]
&-\dfrac{1}{h}\int_\Omega \dfrac{|u_{m-1}|^{q-1}u_{m-1}(u_m+\varepsilon \varphi)-|u_{m-1}|^{q-1}u_{m-1}u_m}{\varepsilon}\,dx \notag \\[3mm]
&+\dfrac{1}{2p}\dfrac{[u_m+\varepsilon\varphi]_{W^{s,p}(\bR^n)}^p-[u_m]^p_{W^{s,p}(\bR^n)}}{\varepsilon} \notag \\[3mm]
&=:I_1+I_2+I_3.
\end{align}
We clearly have
\begin{equation}\label{new eq.b A1}
I_2=-\dfrac{1}{h}\int_\Omega |u_{m-1}|^{q-1}u_{m-1}\varphi\,dx
\end{equation}
while, using the fundamental theorem of calculus, we find from H\"{o}lder's inequality with $u_m \in \mathcal{X}$ that, as $\varepsilon \to 0$,
\begin{align}\label{new eq.c A1}
I_1&=\dfrac{1}{h}\dfrac{1}{q+1}\int_\Omega \dfrac{1}{\varepsilon}\left(\,\int_0^1 \dfrac{d}{d\theta}\left|u_m+\theta \varepsilon \varphi \right|^{q+1}\,d\theta \right)\,dx \notag \\[3mm]
&=\dfrac{1}{h}\int_\Omega \left(\,\int_0^1 \left|u_m+\theta \varepsilon \varphi \right|^{q-1}(u_m+\theta \varepsilon \varphi)\,d\theta \right)\varphi\,dx \to \dfrac{1}{h}\int_\Omega |u_m|^{q-1}u_m\varphi\,dx.
\end{align}
Due to the fundamental theorem of calculus again,
\begin{align}\label{new eq.d A1}
I_3&=\dfrac{1}{2p}\iint_{\bR^n \times \bR^n} \dfrac{1}{\varepsilon}\dfrac{\left|(u_m+\varepsilon \varphi)(x)-(u_m+\varepsilon \varphi)(y)\right|^p-\left|u_m(x)-u_m(y)\right|^p}{|x-y|^{n+sp}}\,dxdy \notag\\[3mm]
&=\dfrac{1}{2p}\iint_{\bR^n \times \bR^n} \dfrac{1}{\varepsilon}\dfrac{1}{|x-y|^{n+sp}}\left(\int_0^1\dfrac{d}{d\theta} \Big|(u_m+\theta\varepsilon \varphi)(x)-(u_m+\theta\varepsilon \varphi)(y)\Big|^p\,d\theta\right)\,dxdy \notag \\[3mm]
&=\dfrac{1}{2}\iint_{\bR^n \times \bR^n} \dfrac{1}{|x-y|^{n+sp}}\left(\int_0^1U_m^{(\theta)}(x,y)\left(\varphi(x)-\varphi(y)\right)\,d\theta\right)\,dxdy
\end{align}
with short-hand notation as well as~\eqref{Uxy}
\[
U_m^{(\theta)}(x,y):=\Big|(u_m+\theta\varepsilon \varphi)(x)-(u_m+\theta\varepsilon \varphi)(y)\Big|^{p-2}\Big((u_m+\theta\varepsilon \varphi)(x)-(u_m+\theta\varepsilon \varphi)(y)\Big).
\]
The integrand of the right-hand side of~\eqref{new eq.d A1} is estimated as
\begin{align*}
&\dfrac{1}{|x-y|^{n+sp}}\left|\int_0^1U_m^{(\theta)}(x,y)\left(\varphi(x)-\varphi(y)\right)\,d\theta\right| \\[3mm]
& \quad \quad \quad \quad \quad \leq 2^{p-1}\left( \dfrac{\left|u_m(x)-u_m(y)\right|^{p-1}|\varphi(x)-\varphi(y)|}{|x-y|^{n+sp}} +\dfrac{\left|\varphi(x)-\varphi(y)\right|^p}{|x-y|^{n+sp}} \right)
\end{align*}
and, from the H\"{o}lder inequality, 
\begin{align*}
&\iint_{\bR^n\times \bR^n} C\left( \dfrac{\left|u_m(x)-u_m(y)\right|^{p-1}|\varphi(x)-\varphi(y)|}{|x-y|^{n+sp}} +\dfrac{\left|\varphi(x)-\varphi(y)\right|^p}{|x-y|^{n+sp}} \right)\,dxdy\\[3mm]
& \quad \quad \quad \quad \quad \leq C\left( [u_m]_{W^{s,p}(\bR^n)}^{p-1}[\varphi]_{W^{s,p}(\bR^n)}+[\varphi]_{W^{s,p}(\bR^n)}^p \right)<\infty.
\end{align*}
By the Lebesgue dominated convergence theorem, we pass to the limit $\varepsilon \to 0$ in~\eqref{new eq.d A1} to conclude that 
\begin{equation}\label{new eq.e A1}
\lim_{\varepsilon \to 0} I_3=\dfrac{1}{2}\iint_{\bR^n \times \bR^n}\dfrac{U_m(x,y)}{|x-y|^{n+sp}}\left(\varphi(x)-\varphi(y)\right)\,dxdy,
\end{equation}
where $U_m(x,y):=\left|u_m(x)-u_m(y)\right|^{p-2}\left(u_m(x)-u_m(y)\right)$ as in~\eqref{Uxy}.
Merging \eqref{new eq.b A1},~\eqref{new eq.c A1} and~\eqref{new eq.e A1} in~\eqref{new eq.a A1}, we therefore obtain from the minimality of $u_m$ of $\mathcal{F}$ in $\mathcal{X}$ that, for all $\varphi \in \mathcal{X}$,
\begin{align*}
0&=\lim_{\varepsilon \to 0}\dfrac{\mathcal{F}(u_m+\varepsilon \varphi)-\mathcal{F}(u_m)}{\varepsilon} \\[3mm]
&=\int_\Omega \dfrac{|u_m|^{q-1}u_m-|u_{m-1}|^{q-1}u_{m-1}}{h}\,\varphi\,dx+\dfrac{1}{2}\iint_{\bR^n \times \bR^n}\dfrac{U_m(x,y)}{|x-y|^{n+sp}}\left(\varphi(x)-\varphi(y)\right)\,dxdy,
\end{align*}
which finishes the proof of Lemma~\ref{existence of NE}.
\end{proof}
\section{Boundedness}\label{Appendix B}
In this Appendix~\ref{Appendix B}, by the additional stronger assumption that the boundedness of the initial datum $u_0$, we conclude slightly stronger assertions compared to ones in the previous section (see Lemmata~\ref{MP} and~\ref{convergence result MP} and Theorem~\ref{mainthm3}).
\subsection{Boundedness}\label{B}
In this Appendix~\ref{B}, we show the boundedness for approximate solutions of \eqref{maineq}.
The test function here is retrieved from~\cite{Alt-Luckhaus}.
\begin{lem}[Boundedness]\label{MP}
Suppose that $u_0 \in W^{s, p}_0 (\Omega) \cap L^\infty (\Omega)$.
Let $\bar{u}_h$,\,$u_h$, \,$\bar{w}_h$, \,$w_h$,\, $v_h$ and $\bar{v}_h$ be a  approximate solution of~\eqref{maineq} defined in~\eqref{approx. sol. for maineq.1} and~\eqref{approx. sol. for maineq.2}.
Then, these solutions are bounded in $\Omega_\infty$:
\begin{equation}\label{MPeq.}
\begin{split}
\sup_{t>0} \|\bar{u}_h(t)\|_{L^\infty(\Omega)}
&\leq \|u_0\|_{L^\infty(\Omega)}, \\
\sup_{t>0} \|u_h(t)\|_{L^\infty(\Omega)}
&\leq \|u_0\|_{L^\infty(\Omega)};
\end{split}
\end{equation}
\begin{equation}\label{MPeq.'}
\sup_{t>0} \|\bar{w}_h(t)\|_{L^\infty(\Omega)},\,\,\sup_{t>0} \|w_h(t)\|_{L^\infty(\Omega)}\leq \|u_0\|_{L^\infty(\Omega)}^{\frac{q+1}{2}} ; 
\end{equation}
\begin{equation}\label{MPeq.''}
\sup_{t>0} \|\bar{v}_h(t)\|_{L^\infty(\Omega)},\,\,\sup_{t>0} \|v_h(t)\|_{L^\infty(\Omega)}
\leq \|u_0\|_{L^\infty(\Omega)}^q.
\end{equation}
\end{lem}
\begin{proof}
The estimations~\eqref{MPeq.}$_2$,~\eqref{MPeq.'} and~\eqref{MPeq.''} simply follow from~\eqref{MPeq.}$_1$ and Lemma~\ref{elementary est.}. Following the argument developed in~\cite[Lemma 3.3, p.164]{Nakamura-Misawa}, we shall prove ~\eqref{MPeq.}$_1$. Put $M:=\|u_0\|_{L^\infty(\Omega)}$ for brevity.
In~\eqref{weak form NE}, choose a testing function $\xi_\delta (u_m)$ with 
\[
\xi_\delta(\sigma)=\dfrac{\sigma}{|\sigma|}\min\left\{1,\,\frac{(|\sigma|-M)_{+}}{\delta} \right\}\quad \textrm{for}\quad \delta>0.
\]
Note that $\xi_\delta (\sigma)$ is clearly bounded and Lipschitz function for $\sigma \in \bR$ and thus, $\xi_\delta (u_m) \in W^{s, p}_0(\Omega) \cap L^\infty (\Omega)$.
\begin{figure}[h]
\begin{tikzpicture}[domain=-3:2, samples=70, very thick,scale=1.1]
\draw[thin, ->] (-5.5,0)--(4.5,0) node[right] {$\sigma$}; 
\draw[thin, ->] (-0.5,-1.7)--(-0.5,1.8) node[above] {\footnotesize$y$}; 
\draw [domain=0:2]plot(\x, {0.6 *\x}); 
\draw (4.5,1.2)node[above] {\footnotesize $y=\xi_\delta(\sigma)$};
\draw (0.3,0)node[below]{\footnotesize $M$};
\filldraw(0,0) circle (0.03);
\filldraw[very thick](-0.5, 0)--(0,0);
\filldraw[very thick](2, 1.2)--(4,1.2);
\filldraw(2,0) circle (0.03);
\draw (2.5,0)node[below]{\footnotesize $M+\delta$};
\draw[densely dotted] (2,0) -- (2,1.2);
\draw[densely dotted] (-0.5,1.2) -- (4,1.2);
\draw (-0.8,1.2)node {\footnotesize$1$};
\draw (-1.1,0)node[above]{\footnotesize $-M$};
\draw [domain=-3:-1]plot(\x, {0.6 *(\x+1)}); 
\filldraw(-1,0) circle (0.03);
\filldraw[very thick](-1, 0)--(-0.5,0);
\filldraw[very thick](-3.0, -1.2)--(-5.0,-1.2);
\filldraw(-3.0,0) circle (0.03);
\draw (-3,0)node[above]{\footnotesize $-(M+\delta)$};
\draw[densely dotted] (-3.0,0) -- (-3.0,-1.2);
\draw[densely dotted] (-0.5,-1.2) -- (-4.5,-1.2);
\draw (-0.8,-1.6)node {\footnotesize$-1$};
\end{tikzpicture}
\caption{Graph of $y=\xi_\delta(\sigma)$}
\end{figure}
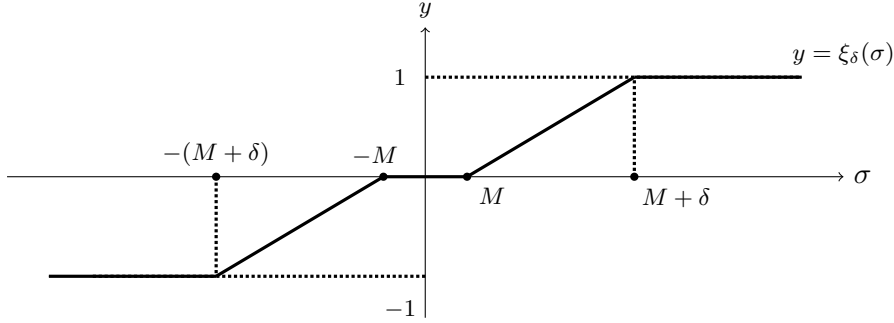

\noindent
Then, for $m=1,2,\ldots$, we have
\begin{align}\label{MP eq.1}
\frac{1}{h}\int_\Omega &\left(|u_{m}|^{q-1}u_{m}-|u_{m-1}|^{q-1}u_{m-1} \right)\xi_\delta(u_m)\,dx\notag \\
&+\frac{1}{2}\iint_{\bR^n \times \bR^n}\frac{U_m(x,y)}{|x-y|^{n+sp}}\big(\xi_\delta(u_m(x))-\xi_\delta(u_m(y))\big) \,dxdy=0.
\end{align}
The integrand in the fractional integral in~\eqref{MP eq.1} is evaluated as follows:
Since $\xi_\delta(u_m(x))-\xi_\delta(u_m(y))$ has the same sign
as $u_m(x)-u_m(y)$ via the monotonicity of $\xi_\delta(\sigma)$,
we observe
\begin{align*}
& U_m(x,y)\bigl(\xi_\delta(u_m(x))-\xi_\delta(u_m(y))\bigr) \\
&=|u_m(x)-u_m(y)|^{p-2} \bigl(u_m(x)-u_m(y)\bigr)
\bigl(\xi_\delta(u_m(x))-\xi_\delta(u_m(y))\bigr) \geq 0
\end{align*}
on $\bR^n \times \bR^n$.

Dropping the fractional term of \eqref{MP eq.1} implies
\begin{align}\label{MP eq.2}
&\int_\Omega \frac{|u_m|^q-|u_{m-1}|^q }{h}\min\left\{1,\,\frac{(|u_m|-M)_+}{\delta} \right\}\,dx \leq 0.
\end{align}
By Lebesgue's dominated convergence theorem
and passage to the limit $\delta \searrow 0$ in~\eqref{MP eq.2} we have 
\begin{equation}\label{MP eq.3}
\int_{\Omega \,\cap\,\{ |u_{m}|>M\}} \frac{|u_m|^q-|u_{m-1}|^q }{h}\,dx \leq 0,
\end{equation}
which is equivalent to
\[ \int_\Omega (|u_m|^q-M^q)_+ \,dx
\leq \int_{\Omega\cap\{|u_m|>M\}} (|u_{m-1}|^q-M^q) \,dx. \]
By neglecting the region such that $|u_{m-1}|\leq M \leq |u_m|$,
the right-hand side is further estimated as
\begin{align*}
\int_{\Omega\cap\{|u_m|>M\}} (|u_{m-1}|^q-M^q) \,dx
&\leq \int_{\Omega\cap\{|u_m|>M\}\cap\{|u_{m-1}|>M\}} (|u_{m-1}|^q-M^q) \,dx \\
&\leq \int_\Omega (|u_{m-1}|^q-M^q)_+ \,dx.
\end{align*}
Summarizing, we have 
\[
\int_\Omega (|u_m|^q-M^q)_+ \,dx \leq \int_\Omega (|u_{m-1}|^q-M^q)_+ \,dx.
\]
Iterating the above display with respect to $m$, we obtain
\[ \int_\varOmega (|u_m|^q-M^q)_+ \,dx
\leq \int_\varOmega (|u_0|^q-M^q)_+ \,dx =0 \]
for all $m=1,2,\ldots$, which in turn implies $|u_{m}|\leq M$ in $\Omega$ for any $m=1,2,\ldots$, finishing the proof.
\end{proof}
By the boundedness~Lemma~\ref{MP}, the convergence result in Lemma~\ref{convergence result} becomes better as the following.
\begin{lem}[Convergence of approximate solutions]\label{convergence result MP}
Let $\bar {u}_h$, \,$u_h$, \,$\bar{w}_h$, \,$w_h$, $\bar{v}_h$ and $v_h$  be the approximate solutions of~\eqref{maineq} defined by~\eqref{approx. sol. for maineq.1} and~\eqref{approx. sol. for maineq.2}, satisfying the convergence in Lemma~\ref{convergence result}. For all finite exponent $\gamma \geq 1$, any bounded domain $K \subset \bR^n$ and every positive number $T<\infty$,
\begin{align}
u_h,\,\bar{u}_h \to u \quad &\textrm{strongly\,\,in}\,\,L^\gamma(K_T),
\label{conv. MP1} \\
w_h,  {\bar w}_h \to |u|^{\frac{q - 1}{2}} u
\quad &\textrm{strongly\,\,in}\,\,L^\gamma(K_T),
\label{conv. MP2} \\
v_h,\,\bar{v}_h \to |u|^{q-1}u \quad &\textrm{strongly\,\,in}\,\,L^\gamma(K_T)
\label{conv. MP3} 
\end{align}
as $h\searrow 0$.
\end{lem}
\begin{proof}
From the boundedness~\eqref{MPeq.} and~\eqref{conv. 3} it follows that
\begin{equation}\label{MP, proof}
\sup_{0<t<T}\|u(t)\|_{L^\infty(K)} \leq \|u_0\|_{L^\infty(\Omega)}.
\end{equation}
The boundedness~\eqref{MPeq.},~\eqref{MP, proof} and the convergence~\eqref{conv. 2} with the H\"older inequality gives~\eqref{conv. MP1}$_1$. The validity of other statements~\eqref{conv. MP1}$_2$,~\eqref{conv. MP2} and~\eqref{conv. MP3} is shown by use of the boundedness Lemma~\ref{MP},~\eqref{conv. a} and~\eqref{conv. 4}.
\end{proof}
\subsection{A regularity of $\partial_t\left(|u|^{q-1}u\right)$ in the case $q \geq 1$}\label{MP, theorem}
Here we state the regularity on the time-derivative, obtained from the boundedness Lemma~\ref{MP}.
\begin{thm}\label{mainthm3}
Let $q \geq 1$. Suppose that the initial datum $u_0 \in W^{s, p}_0 (\Omega) \cap L^\infty(\Omega)$. Then, the weak solution $u$ of~\eqref{maineq} obtained in Theorem~\ref{maineq} satisfies $\partial_t (|u|^{q - 1} u) \in L^2 (\Omega_\infty)$ and the solution $u$ itself satisfies~\eqref{maineq} almost everywhere in $\Omega_\infty$, as an $L^2$-function.
\end{thm}
\medskip

We first derive the $L^2$-estimate of the time-derivative:
\begin{lem}\label{L2-est'}
Let $q \geq 1$ and $v_h$ be the approximate solutions of~\eqref{maineq} defined by~\eqref{approx. sol. for maineq.2}.  Assume that $u_0 \in W^{s, p}_0 (\Omega) \cap L^\infty (\Omega)$. Then the time derivatives of approximate solutions $\{\partial_t v_h\}$ are bounded in $L^2 (\Omega_T)$ for any positive $T < \infty$, 
\begin{equation}\label{(iii)}
\iint_{\Omega_T} |\partial_t v_h|^2\,dxdt \leq C \|u_0\|_{L^\infty(\Omega)}^{q-1} [u_0]_{W^{s,p}(\bR^n)}^p
\end{equation}
with a positive constant $C=C(p,q)$.
\end{lem}
\begin{proof}[\normalfont \textbf{Proof of Lemma~\ref{L2-est'}}]
From the energy estimate~\eqref{energy est. NE3} in Lemma~\ref{energy est. NE} and~\eqref{MPeq.} in Lemma~\ref{MP}, we obtain
\begin{align*}
\sum_{m=1}^N h \int_\Omega |\partial_t v_h|^2\,dx
&\leq C \sum_{m=1}^N h \int_\Omega
\left(|u_m| + |u_{m - 1}|\right)^{2 (q - 1)} 
\left|\frac{u_m-u_{m-1}}{h}\right|^2\,dx \\[2mm]
&\leq C \sum_{m=1}^N h
\left(%
\|u_m\|_{L^\infty(\Omega)}^{q-1}+\|u_{m-1}\|_{L^\infty(\Omega)}^{q-1}
\right) \\
&\quad \quad \quad \quad\quad \quad \quad\quad\times \int_\Omega \left(|u_m| + |u_{m - 1}|\right)^{q - 1}\left|\frac{u_m-u_{m-1}}{h} \right|^2\,dx \\[2mm]
&\!\!\!\stackrel{\eqref{MPeq.}}{\leq}C\|u_0\|_{L^\infty(\Omega)}^{q-1}\int_0^{t_N}\int_{\Omega} \left(|{\bar u}_h (x,t)| 
+ |{\bar u}_h (x,t-h)|\right)^{q - 1}
|\partial_t u_h(x,t) |^2\,dxdt \\[2mm]
&\!\!\!\!\stackrel{\eqref{energy est. NE3}}{\leq} C\|u_0\|_{L^\infty (\Omega)}^{q-1}[u_0]_{W^{s,p}(\bR^n)}^p,
\end{align*}
which finishes the proof of~\eqref{(iii)}.
\end{proof}
We report the proof of Theorem~\ref{mainthm3}.
\begin{proof}[\normalfont \textbf{Proof of Theorem~\ref{mainthm3}}] Let $q \geq 1$. By~\eqref{(iii)} in Lemma~\ref{L2-est'}, $\{\partial_t v_h\}_{h>0}$ is bounded in $L^2 (\Omega_\infty)$ and thus, there are a subsequence $\{\partial_t v_h\}_{h>0}$ (also labelled with $h$) and a limit function $\vartheta \in L^2(\Omega_\infty)$ such that, as $h \searrow 0$,
%
%
\begin{equation}\label{MP, proof2}
\partial_t v_h \to \vartheta \quad \textrm{weakly\,\,in}\,\,\,L^2(\Omega_\infty).
\end{equation}
By the convergence~\eqref{conv. MP3} in Lemma~\ref{convergence result MP} we pass to the limit as $h \searrow 0$ in the identity
\begin{equation*}
\iint_{\Omega_\infty}\partial_t v_h\cdot \varphi\,dxdt =-\iint_{\Omega_\infty}v_h\,\partial_t \varphi\,dxdt
\end{equation*}
for every $\varphi \in C^\infty_0(\Omega_\infty)$ to obtain from~\eqref{vh} that
\begin{equation}\label{MP, proof3}
\vartheta=\partial_t(|u|^{q-1}u)\quad \textrm{in}\,\,L^2(\Omega_\infty).
\end{equation}
As a result, by the convergence~\eqref{conv. MP3} in Lemma~\ref{convergence result MP} with~\eqref{(iii)} in Lemma~\ref{L2-est'}, there holds that
\begin{equation*}
\left\|\partial_t (|u|^{q-1}u)\right\|_{L^2(\Omega_T)}^2 \leq C\|u_0\|_{L^\infty(\Omega)}^{q-1}[u_0]_{W^{s,p}(\bR^n)}^p,
\end{equation*}
which finishes the proof of Theorem~\ref{mainthm3}.

\end{proof}
\section{A space-time fractional Sobolev inequality}\label{Appendix C}
In this appendix, 
we accomplish in Lemma~\ref{FSineq.II} a general inequality interpolating by the Gagliardo seminorm in time-space domain in the special case where the function is weakly differentiable in time direction. The
inequality in Lemma~\ref{FSineq.II} is used essentially as the key ingredient for Lemma~\ref{convergence result}

The following inequality gives an interpolation
between fractional Sobolev semi-norms and the integral of time derivative.

\begin{lem}[Space-time fractional Sobolev inequality]\label{FSineq.II}
Let $K \subset \bR^n$ be a bounded domain and let $I:=(0,T)$ with $0<T<\infty$. Let $s^\prime, \bar{s}$ be two exponents obeying $0<s^\prime<\bar{s} <1$. Then it holds 
\begin{equation}\label{FSineq.2}
[f]_{W^{s^\prime,1}(K_T)} \leq C(n,n,s^\prime, \bar{s}) \left(\|\partial_t f\|_{L^1(K_T)}+\int_I[f(\cdot, t)]_{W^{\bar{s},1}(K)}\,dt\right),
\end{equation}
where the integrals of a function $f$ in the right-hand side is assumed to be finite.
\end{lem}
\begin{proof}
%
To begin, we divide the Gagliardo semi-norm into two terms:
\begin{align}\label{B2 eq.1}
[f]_{W^{s^\prime, 1}(K_T)} &\leq \iint_{I \times I}\iint_{K \times K}\dfrac{|f(x,t)-f(x,t^\prime)|}{\left(\sqrt{|x-x'|^2+(t-t')^2}\right)^{(n+1)+s^\prime}}\,dxdx'dtdt'  \notag\\[3mm]
&+ \iint_{I \times I}\iint_{K \times K}\dfrac{|f(x,t^\prime)-f(x^\prime,t^\prime)|}{\left(\sqrt{|x-x'|^2+(t-t')^2}\right)^{(n+1)+s^\prime}}\,dxdx'dtdt'\notag\\[3mm]
&=:I+II.
\end{align}
By the assumption it plainly holds $1+(s^\prime-\bar{s}) >0$.
%
%
Under this condition, the integration $I$ is simply estimated as
\begin{align}\label{B2 eq.2}
I&=\iint_{I\times I}\iint_{K \times K} \dfrac{|f(x,t)-f(x,t^\prime)|}{\left(\sqrt{|x-x'|^2+(t-t')^2}\right)^{1+\bar{s}}} \times \frac{dxdx'dtdt' }{\left(\sqrt{|x-x'|^2+(t-t')^2}\right)^{n+(s^\prime-\bar{s})}} \notag\\[3mm]
&\leq \int_{K}\left(\,\,\iint_{I\times I}\dfrac{|f(x,t^\prime)-f(x,t^\prime)|}{|t-t'|^{1+\bar{s}}} \,dtdt' \right)\left(\,\int_{K_{x^\prime}}\frac{dx^\prime}{|x-x'|^{n+(s^\prime-\bar{s})}}  \right) \,dx \notag\\[3mm]
&= \int_{K}\left[\,\,\int_I \left(\,\,\int_{\{\tau \in (-T,T)\,: \,t+\tau \in I\}}\dfrac{|f(x,t^\prime)-f(x,\tau+t)|}{|t-t'|^{1+\bar{s}}} \,d\tau \right)dt' \right]\left(\,\int_{K_{x^\prime}}\frac{dx^\prime}{|x-x'|^{n+(s^\prime-\bar{s})}}  \right) \,dx.
\end{align}
Due to the fundamental theorem of calculus, 
\begin{align*}
\big|f(x,t)-f(x,\tau+t)\big| = \left|\int_0^1 \dfrac{d}{d\sigma} f(x,t+\sigma\tau) \,d\sigma \right| &\leq |\tau| \int_0^1 \Big|\partial_t f(x,t+\sigma \tau)\Big|\,d\tau,
\end{align*}
which, together with~\eqref{B2 eq.1} and Fubini's theorem, in turn implies that
\begin{align}\label{B2 eq.3}
I &\leq \int_{K}\int_0^1\left[\,\,\int_{-T}^T\dfrac{1}{|\tau|^{\bar{s}}} \left(\,\,\int_{\{t \in I\,: \,t+\tau \in I\}}\Big|\partial_tf(x, t+\sigma \tau)\Big| \,dt \right)d\tau\right]\left(\,\int_{K_{x^\prime}}\frac{dx^\prime}{|x-x'|^{n+(s^\prime-\bar{s})}}  \right) \,d\sigma dx \notag \\[3mm]
&\leq \left(\,\,\int_{-T}^T\dfrac{1}{|\tau|^{\bar{s}}} d\tau \right) \int_K\|\partial_tf(x,\cdot)\|_{L^1(I)}\left(\,\int_{K_{x^\prime}}\frac{dx^\prime}{|x-x'|^{n+(s^\prime-\bar{s})}}  \right) \,dx.
\end{align}
Via change of variable $\rho=|x-x'|$ we have
\begin{align}\label{B2 eq.4}
\int_{K_{x^\prime}}\frac{dx^\prime}{|x-x'|^{n+(s^\prime-\bar{s})}}&\leq C(n)\int_0^{\diam K} \frac{d\rho}{\rho^{1+(s^\prime-\bar{s})}} \notag\\[2mm]
&=\frac{(\diam K)^{\bar{s}-s^\prime}}{\bar{s}-s^\prime}<\infty.
\end{align}
Analogously, by the condition $0 <\bar{s}<1$, the integral $\int_{-T}^T\frac{1}{|\tau|^{\bar{s}}} d\tau$ appearing in~\eqref{B2 eq.3} takes a finite value depending $T$ only. Combining this with~\eqref{B2 eq.3} and~\eqref{B2 eq.4}, we get
\begin{align}\label{B2 eq.5}
I &\leq C(n,s^\prime,\bar{s},T)\int_K\|\partial_tf(x,\cdot)\|_{L^1(I)}\,dx \notag\\
&=C(n,s^\prime,\bar{s},T)\|\partial_tf\|_{L^1(K \times I)}.
\end{align}

We now turn to evaluate $II$. The integration $II$ is estimated as
\begin{align}\label{B2 eq.6}
II& \leq \iint_{I\times I}\iint_{K \times K} \dfrac{|f(x,t^\prime)-f(x^\prime,t^\prime)|}{\left(\sqrt{|x-x'|^2+(t-t')^2}\right)^{n+\bar{s}}} \times \frac{dxdx'dtdt' }{\left(\sqrt{|x-x'|^2+(t-t')^2}\right)^{1+s^\prime-\bar{s}}}  \notag\\[3mm]
&\leq \iint_{I \times I}\iint_{K \times K} \dfrac{|f(x,t^\prime)-f(x^\prime,t^\prime)|}{|x-x'|^{n+\bar{s}}}\times \frac{dxdx'dtdt' }{|t-t'|^{1+(s^\prime-\bar{s})}}  \notag\\[3mm]
&=\int_{I}\left(\,\int_{I_t}\dfrac{dt}{|t-t'|^{1+(s^\prime-\bar{s})}} \right)[f(\cdot, t^\prime)]_{W^{\bar{s},1}(K)}\,dt^\prime \notag\\[3mm]
&\leq C \int_{I}\left(\,\,\int_{-T}^{T}\dfrac{d\tau}{|\tau|^{1+(s^\prime-\bar{s})}} \right)\,[f(\cdot, t^\prime)]_{W^{\bar{s},1}(\bR^n)}\,dt^\prime \notag\\[3mm]
&=C(n,s,\bar{s},T)\int_{I}[f(\cdot, t^\prime)]_{W^{\bar{s},1}(K)}\,dt^\prime,
\end{align}
where, in the fourth line, $\displaystyle \int_{-T}^{T}\dfrac{d\tau}{|\tau|^{1+(s^\prime-\bar{s})}} =\frac{2T^{\bar{s}-s^\prime}}{\bar{s}-s^\prime}$ is used.
Merging~\eqref{B2 eq.5}, ~\eqref{B2 eq.6} in~\eqref{B2 eq.2}, we end up with the conclusion~\eqref{FSineq.2}.
\end{proof}
%

%

\end{document}